\documentclass[10pt,a4paper]{article}
\usepackage[utf8]{inputenc}
\usepackage[english]{babel}
\usepackage{multicol,graphicx,color}
\usepackage{pslatex}
\usepackage{authblk}
\usepackage{amsthm}
\usepackage{amsmath}
\usepackage{amssymb}
\usepackage{latexsym}
\usepackage{lscape}
\usepackage{epsfig}
\usepackage{amsfonts}
\usepackage{mathrsfs,bbm}
\usepackage[
hmarginratio={1:1},     % equal left and right margins
vmarginratio={1:1},     % equal top and bottom margins=
textwidth=15cm,        % new text width
textheight=21cm,
heightrounded,          % always useful
]{geometry}

\usepackage{tikz}

\makeatletter

\newtheorem{theorem}{Theorem}[section]
\newtheorem{corollary}[theorem]{Corollary}

\newtheorem{lemma}[theorem]{Lemma}
\newtheorem{proposition}[theorem]{Proposition}
\newtheorem{remark}[theorem]{Remark}

\theoremstyle{definition} \theoremstyle{remark}
\numberwithin{equation}{section}

\newcommand{\G}{\mathcal{G}}
\newcommand{\R}{\mathbb{R}}
\newcommand{\N}{\mathbb{N}}
\newcommand{\B}{\mathcal{B}}
\newcommand{\U}{\mathcal{U}}
\newcommand{\K}{\mathcal{K}}

\newcommand{\p}{\mathcal{P}}
\newcommand{\s}{\mathbb{S}}

\newcommand{\cl}{\mathcal{L}}

\newcommand{\la}{\lambda}

\newcommand{\incident}{\succ}
\newcommand{\edge}{{\mathrm e}} % edge
\newcommand{\vv}{{\mathrm{v}}}

\newcommand{\intd}{\,\mathrm{d}}
\newcommand{\intervalco}[1]{\mathopen[#1\mathclose)}

\begin{document}

\title{Sign-changing prescribed mass solutions for $L^2$-supercritical NLS on compact metric graphs
}

%\author{}
\author{Louis Jeanjean\footnote{louis.jeanjean@univ-fcomte.fr}}  \affil{{\small Universit\'e Marie et Louis Pasteur, CNRS, LmB (UMR 6623), Besan\c{c}on F-25000, France}}

\author{Linjie Song\footnote{songlinjie18@mails.ucas.edu.cn}}  \affil{{\small Tsinghua University, Department of Mathematical Sciences, Beijing 100084, China}}

	%Université Marie et Louis Pasteur, CNRS, LmB (UMR 6623), F-25000 Besançon, France.
\date{}	
\maketitle

\begin{abstract}
	This paper is devoted to the existence of multiple sign-changing solutions of prescribed mass for a mass-supercritical nonlinear Schr\"odinger equation set on a compact metric graph. In particular, we obtain, in the supercritical mass regime, the first multiplicity result for prescribed mass solutions on compact metric graphs. As a byproduct, we prove that any eigenvalue of the associated linear operator is a bifurcation point. Our approach relies on the introduction of a new kind of link and on the use of gradient flow techniques on a constraint. It can be transposed to other problems posed on a bounded domain.
\end{abstract}

\medskip

{\small \noindent \textbf{Key Words:} Constrained functionals; $L^2$-supercritical; compact metric graph; sign-changing critical points, bifurcation phenomena.}

\medskip

{\small \noindent \textbf{Acknowledgements:} This work has been carried out in the framework of the Project NQG (ANR-23-CE40-0005-01), funded by the French National Research Agency (ANR). It was completed when the second author was visiting the Universit\'e Marie et Louis Pasteur, LmB (UMR 6623). Linjie Song is supported by "Shuimu Tsinghua Scholar Program" and by "Postdoctoral Fellowship Program of CPSF" (GZB20230368), and funded by "China Postdoctoral Science Foundation" (2024T170452). The first author thanks D. Galant and S. Le Coz, for useful remarks.}

\medskip
{\small \noindent \textbf{Data availability:} Data sharing is not applicable to this article as no datasets were generated or analysed during the current study.}

\medskip
{\small \noindent \textbf{Ethical Statement}
	
\noindent \text{Conflict of interest:} The authors have no relevant financial or non-financial interests to disclose.}

\section{Introduction and main results}

In \cite{CJS}, the first author of this article and his collaborators studied the existence of non-constant critical points for the $C^2$ \emph{mass supercritical} NLS energy functional $E(\cdot,\G): H^1(\G) \to \R$ defined by
\begin{align} \label{eqnerf}
	E(u,\G) = \frac12 \int_\G|u'|^2dx - \frac1p \int_\G|u|^pdx, \quad p > 6,
\end{align}
under the mass constraint
\begin{align} \label{eqmass}
	\int_\G|u|^2dx = \mu > 0,
\end{align}
where $\G$ is a \emph{compact metric graph}. Such critical points, also called bound states, solve the stationary nonlinear Schr\"odinger equation (NLS) on $\G$
\begin{align} \label{eqequation}
	-u'' + \la u = |u|^{p-2}u
\end{align}
for some Lagrange multiplier $\la$, coupled with Kirchhoff condition at the vertices (see \eqref{eqequwithkc} below). In turn, the solutions of \eqref{eqequation} give rise to time-dependent NLS standing waves on $\G$,
\begin{align} \label{timenls}
\mathrm{i}\, \partial_t \psi(t,x)
= -\partial_{xx} \psi(t,x)-|\psi(t,x)|^{p-2} \psi(t,x),
\end{align}
via the ansatz $\psi(t,x) = e^{\mathrm{i} \lambda t}u(x)$. The constraint \eqref{eqmass} is meaningful from a dynamics perspective as the mass (or charge), as well as the energy, is conserved by the NLS flow.

One of the main physical motivations to consider \eqref{timenls} on metric graph is the study of propagation of optical pulses in nonlinear optics, or of matter waves in the theory of Bose-Einstein condensates and in branched structures, such as $T$-junctions or $X$-junctions. We refer interested readers to \cite{AST3, BK, KNP, LLMetc, Noja} and to the references therein for more details. 
In addition, there are some interesting new mathematical features in metric graphs compared to the Euclidean case, and thus, the problem of existence of bound states on metric graphs attracted a lot of attention in the past decade, mainly in the \emph{subcritical} or \emph{critical regimes} (i.e., $p \in (2,6)$ or $p = 6$ respectively) in which the energy functional $E(\cdot,\G)$ is bounded from below and coercive on the mass constraint (for $\mu >0$ small when $p=6$). 

In these cases, the problem mainly studied is that of the existence of a global minimizer of the energy functional under the mass constraint.  We refer to \cite{ACFN, AST, AST2, AST4} for cases where the graph is non-compact and to \cite{CDS, Dov} where it is compact. See also \cite{DeDoGaSe, DeDoGaSeTr, DT, NP, PS, PSV, ST} for strictly related issues (problems with localized nonlinearities, combined nonlinearities, existence of critical points in absence of ground states).

In striking contrast, the supercritical regime ($p > 6$) on general graphs is less touched. Before \cite{CJS}, we are only aware of \cite{Ard} in which the assumptions allow the analysis to be reduced to the study of minimizing sequences living in a bounded subset of the constraint. In the setting of \cite{CJS} where $\G$ is compact, there always exists a constrained constant (trivial) critical point of $E(\cdot,\G)$ obtained by taking the constant function $\kappa_\mu := (\mu/\ell)^{1/2}$, where $\ell$ denotes the total length of $\G$. In  \cite{CJS} the authors showed that 
for any $0 < \mu < \mu_1$, where 
\begin{equation}\label{defmu1}
\mu_1:= \ell\left(\frac{\la_2(\G)}{p-2}\right) ^\frac2{p-2},
\end{equation}
there is a positive non-constant critical point of $E(\cdot,\G)$ in $H^1_\mu(\G)$, which is at a mountain pass level.

Note that if one tries to apply the approach of \cite{Ard} to a compact graph, it is not possible to know whether the resulting local minimizer is constant or non-constant. For this reason, the existence of non-trivial bound states in the supercritical case for general (compact) metric graphs can be seen as first essentially touched in \cite{CJS}. Later, in \cite{CGJT}, the first author with his co-authors  proved the existence of infinitely many bound states with any prescribed mass for $L^2$-supercritical nonlinear Schr\"odinger (NLS) equations with localized nonlinearities on non-compact metric graphs, coupled with the Kirchhoff conditions at the vertices. Unfortunately, the method used in \cite{CGJT} fails to work on compact graph  and it is left as an open problem to prove the multiplicity, even the existence of just two non-trivial critical points of $E(\cdot,\G)$ under the mass constraint for compact $\G$, see \cite[Remark 1.5]{CGJT}. The main motivation of this paper is to solve this open question.

\subsection*{Basic notations and main result}

For any graph, we write
$\G = (\mathcal{E}, \mathcal{V})$, where $\mathcal{E}$ is the set of
edges and $\mathcal{V}$ is the set of vertices. Each bounded edge
$\edge$ is identified with a closed bounded interval $I_\edge = [0,\ell_\edge]$ (where $\ell_\edge$ is the length of $\edge$), while each unbounded edge is identified with a closed half-line $I_\edge = \intervalco{0,+\infty}$. The length of the shortest path between points provides a natural metric (whence a topology and a Borel structure) on $\G$. A metric graph is compact if and only if it has a finite number of edges, and none of them is unbounded. We will assume throughout the paper that it is the case.

A function $u: \G \to \R$ is identified with a vector of functions $\{u_{\edge}\}_{\edge \in \mathcal{E}}$, where each $u_{\edge}$ is defined on the corresponding interval $I_\edge$ such that $u|_{\edge}=u_{\edge}$.  Endowing each edge with the Lebesgue measure, one can define $\int_\G u(x) \intd x$ and the space $L^p(\G)$ in a natural way, with norm
\begin{equation*}
  \|u\|_{L^p(\G)}^p
  = \sum_{\edge \in \mathcal{E}} \|u_\edge\|_{L^p(\edge)}^p.    
\end{equation*}
The Sobolev space $H^1(\G)$ consists of the set of continuous
functions $u: \G \to \R$ such that
$u_{\edge} \in H^1(\mathring{\edge})$ for every edge
$\edge$; the norm
in $H^1(\G)$ is defined as
\begin{equation*}
  \|u\|_{H^1(\G)}^2
  = \sum_{\edge \in \mathcal{E}} \left(\|u_{\edge}'\|_{L^2(\edge)}^2 + \|u_{\edge}\|_{L^2(\edge)}^2\right).
\end{equation*}
More details can be found in \cite{AST, AST2, BK}.

We aim to prove the multiplicity of non-constant critical points of the functional $E(\cdot,\mathcal{G})$, defined in \eqref{eqnerf}, constrained on the $L^2$-sphere
\begin{equation*}
	H_{\mu}^1(\mathcal{G})
	:= \bigl\{u\in H^1(\mathcal{G}):
	\int_{\mathcal{G}}|u|^2 dx=\mu \bigr\}.
\end{equation*}
If 
$u\in H^1_{\mu}(\mathcal{G})$ is such a critical point, it is standard to show that there exists a Lagrange multiplier $\lambda\in \mathbb{R}$ such that $u$ satisfies the following problem:
\begin{equation}\label{eqequwithkc}
	\begin{cases}
		-u''+\lambda u=|u|^{p-2}u
		&\text{on every edge } \edge \in \mathcal{E},\\[1\jot]
		\displaystyle
		\sum_{\edge \incident \vv} u_{\edge}'(\vv)=0
		&\text{at every vertex } \vv \in \mathcal{V},
	\end{cases}
\end{equation}
where $\edge \incident \vv$ means that the edge $\edge$ is incident at
$\vv$, and the notation $u_{\edge}'(\vv)$ stands for
$u'_{\edge}(0)$ or $-u'_{\edge}(\ell_{\edge})$, according to whether
the vertex $\vv$ is identified with $0$ or $\ell_{\edge}$ (namely, the
sum involves the derivatives away from the vertex $\vv$). The
second equation is the so-called \emph{Kirchhoff boundary condition}. Note that at external vertices, namely vertices which are reached by a unique edge, the Kirchhoff conditions reduce to purely Neumann conditions. 
%Finally, notice that the positive constant function $\kappa_\mu := (\mu/\ell)^{1/2}$ trivially satisfies \eqref{eqequwithkc}, for $\la = (\mu/\ell)^{(p-2)/2}$.

%Recall that it was shown in \cite{CJS} that for any $0 < \mu < \mu_1$, there is a positive non-constant critical point of $E(\cdot,\G)$ constrained on $H^1_\mu(\G)$, which is of mountain pass type. It is open and interesting to find more critical points of $E(\cdot,\G)$ with the mass constraint \eqref{eqmass}, by reducing $\mu$ if necessary.

Our strategy to find multiple non-constant solutions is to look for sign-changing critical points of $E(\cdot,\G)$ in $H^1_\mu(\G)$. Of course, any sign-changing critical point is non-constant and is different from the positive one found in \cite[Theorem 1.1]{CJS}. Our first main result is as follows.

\begin{theorem} \label{thmsc}
	Let $\G$ be any compact metric graph, and $p > 6$. There exists $\mu_2 > 0$ depending on $\G$ and on $p$ such that, for any $0 < \mu < \mu_2$, $E(\cdot,\G)$ has a sign-changing (non-constant) critical point constrained on $H^1_\mu(\G)$.
	
	Furthermore, for any $0 < \mu < \mu_2$, $E(\cdot,\G)$ has at least two pairs of non-constant critical points $\pm u_1, \pm u_2$ constrained on $H^1_\mu(\G)$, $u_1$ (respectively, $-u_1$) is positive (respectively, negative) and at a mountain pass level, and $u_2$ is sign-changing.
\end{theorem}

\begin{remark}
	Theorem \ref{thmsc} is not a perturbation result, the value $\mu_2$ will not be obtained by any limit process, and can be explicitly estimated. We refer to the proof of Theorem \ref{thmsc} in Section \ref{secproof} for more details.
\end{remark}

The novelty of Theorem \ref{thmsc} lies in the existence of sign-changing critical points. A main tool in our proof is a new kind of link contained in an open bounded subset $\B_\rho^\mu$ of $H^1_\mu(\G)$. We note that a similar idea was used by the second author and his co-author in \cite{SZ} to find sign-changing solutions with prescribed mass for the Br\'ezis-Nirenberg problem. 
The subset $\B_\rho^\mu$ is given
$$\B_\rho^\mu = \bigl\{u \in H^1_\mu(\G): \int_\G|u'|^2dx < \rho\bigr\}.$$
All our analysis will be developed in  $\B_\rho^\mu$ where Palais-Smale sequences are, trivially, bounded, something which may not be true outside of $\B_\rho^\mu$. Since the graph $\G$ is compact, the Palais-Smale condition thus holds on the closure of $\B_\rho^\mu$, see Corollary \ref{corpscon}. For $\mu >0$ sufficiently small, we fix a convenient $\rho >0$ and define inside $\B_\rho^\mu$ a min-max structure such that the min-max level is strictly below the infimum of the energies of functions on $\partial \B_\rho^\mu$; see Proposition \ref{propinf} and Section \ref{seclink} for more details. Thus, if we find a critical point at the min-max level, it will not belong to $\partial \B_\rho^\mu$ and will be a critical point constrained on $H^1_\mu(\G)$ without additional constraints. 

In unconstrained problems, while the existence of critical points with sign change has been widely studied (see, for example, \cite{BLW, LS, LWW, RTZ, SchZ1, SchZ, ZZ, Zou}), it is known to be more difficult than simply finding a critical point without any sign requirement. The existence of sign-changing constrained critical points has been much less studied. 
Regarding $L^2$-constrained problems set on $\R^N$, we mention \cite{JL1}, where the existence of sign-changing nonradial solutions with prescribed mass, for some NLS on $\R^N$ with $N \geq 4$ in a mass subcritical case, was first observed. Related results were obtained in \cite{JL2} in a mass supercritical setting.
We remark that the methods in \cite{JL1,JL2} strongly depend on the symmetry of $\R^N$ and the Pohozaev manifold was used in the mass supercritical case, for which they fail to work for $L^2$-supercritical NLS on bounded domains and on compact metric graphs. For constrained problems on bounded domains, in a recent work \cite{DeDoGaSe2} an action approach was used to establish the existence of one sign-changing solution with prescribed mass. 
Notably, the powerful descending flow technique, which has been widely developed for finding sign-changing solutions in unconstrained problems, has been little used in a constrained $L^2$ case. A key step here is to find invariant subsets with respect to some negative gradient flow on the constraint. In \cite{TaTe}, see also \cite{FeZhZhZh}, this was successfully implemented in the problems considered. If one hopes to develop a somehow general approach, which up to our knowledge is missing, a first observation is that the classical Br\'ezis-Martin result cannot be used directly in the presence of a mass constraint.

In Section \ref{secinv}, we extend the Br\'ezis-Martin result (see Proposition \ref{propbm}) to an abstract setting in which our constrained problem is included; see Proposition \ref{propbmmanifold} and Corollary \ref{corbm} for more details.
 We also provide Lemma \ref{lemconvex} to check a key assumption used by Proposition \ref{propbmmanifold}. We believe Proposition \ref{propbmmanifold}, Corollary \ref{corbm}, and Lemma \ref{lemconvex} are of independent interest. They constitute a general framework to prove that some subset is invariant with respect to a flow on a manifold. This framework can be used to establish the existence of sign-changing bound states with prescribed mass in some other problems. 
In fact, it has already been used in \cite{SZ}.

The gradient $\nabla E(\cdot, \G)$ of $E(\cdot,\G)$ as a constrained functional on $H^1_\mu(\G)$ belongs to the tangent space of $H^1_\mu(\G)$ at the point $u \in H^1_\mu(\G)$ and is given by
$\nabla E(u,\G) = u - (-\frac{d^2}{dx^2} +1)^{-1}(|u|^{p-2}u + \la_uu).$
In view of Lemma \ref{lemconvex}, to check, in our problem, the assumptions in Proposition \ref{propbmmanifold} and Corollary \ref{corbm}, we need to determine the sign of $\la_u$. This is a main difference with the unconstrained case where $\la_u$ is prescribed and independent of $u$.

Our process of searching for sign-changing critical points can be continued after finding the first one. In fact, we can obtain as many sign-changing critical points of $E(\cdot,\G)$ in $H^1_\mu(\G)$ as we want at the expense of possibly reducing $\mu$. Our second multiplicity result is as follows.

\begin{theorem} \label{thmsc2}
	Let $\G$ be any compact metric graph, and $p > 6$. For any $j \in \N$, there exists $\mu_j > 0$ depending on $\G$ and on $p$ such that, for any $0 < \mu < \mu_j$, $E(\cdot,\G)$ has at least $j$ pairs of non-constant critical points $\pm u_1, \pm u_2, \cdots, \pm u_j$ constrained on $H^1_\mu(\G)$. Furthermore, $u_1$ (respectively, $-u_1$) is positive (respectively, negative) and the others are sign-changing.
\end{theorem}

%\begin{remark}
%	\cite[Theorem 1.1]{CJS} and Theorem \ref{thmsc} in this paper are special cases of Theorem \ref{thmsc2} for $j =1$ and $j = 2$ respectively.
%\end{remark}

The proof of Theorem \ref{thmsc2} is quite similar to the one of Theorem \ref{thmsc}. We combine discussions on min-max structures, descending flow techniques, and  we establish the invariance of some subsets of $H^1_\mu(\G)$ with respect to some negative gradient flow. An additional difficulty is to make sure that our approach provides distinct critical points. This is done by estimating the energy levels of the sign-changing critical points, see step 4 in the proof of Theorem \ref{thmsc2}.
\smallskip

In the last part of the paper, we discuss the limiting behaviour of the solutions obtained in Theorem \ref{thmsc2}, as the mass $\mu \to 0$.  First we observe that all our solutions are obtained inside the sets $\B^{M_2} \subset H^1_{\mu}(\G)$ on which the $H^1(\G)$ norm is bounded uniformly with respect to $\mu >0$. This implies that the norm $H^1(\G)$ and the energy of all our solutions are uniformly bounded. This contrasts sharply with \cite{CJS} where the norm and energy of the solution found in \cite[Theorem 1.1]{CJS} go to infinity when $\mu \to 0$. Actually, it is possible to describe very precisely what happens when $\mu \to 0$.  Following the terminology introduced by C. A. Stuart \cite{St1,St2}, we say that $\lambda^* \in \R$ is a bifurcation point 
of \eqref{eqequwithkc} if there exists a sequence $\{ (\lambda_k, u_k) \} \subset \R \times H^1(\G)$ such that $u_k \not \equiv 0$ solves \eqref{eqequwithkc} with $\lambda = \lambda_k$, $-\lambda_k \to \lambda^*$  and $||u_k||_{H^1(\G)} \to 0$. 	
\smallskip

Let $0=\la_1(\G) < \la_2(\G) \leq \la_3(\G) \leq \cdots $ be the eigenvalues of $-\frac{d^2}{dx^2}$ on the compact metric graph $\G$ with Kirchhoff condition at the vertices.
Based on the process of building our solutions, we obtain

\begin{theorem} \label{thmsc3}
Let  $\G$ be any compact metric graph, and $p > 6$. Then the set of bifurcation point of \eqref{eqequwithkc} is $\{\la_j(\G): j \in \N \}$.
\end{theorem}
From classical bifurcation theory, any bifurcation point of \eqref{eqequwithkc} must be an eigenvalue, namely belong to $\{\la_j(\G): j \in \N \}$, and any simple eigenvalue is a bifurcation point. Theorem \ref{thmsc3} states that any eigenvalue, regardless of its multiplicity, is a bifurcation point. This result is of independent interest. \smallskip
%\textcolor{red}{Note that a related bifurcation phenomena was also observed in \cite{SZ}, in the frame of the Brézis-Nirenberg problem.}\smallskip

Finally, let us point out that if the problem here is set on a metric graph what is really essential in our approach is that the operator $-\frac{d^2}{dx^2}$, considered with the Kirchhoff condition, has a discrete spectrum (due here to the compactness of the graph). Our approach could be used to handle other type of problems set on a compact domain, typically  NLS equations on a bounded domain $\Omega \subset \R^N$. \smallskip

The paper is organized as follows. In Section \ref{secpre}, we give some preliminary results and prepare the introduction of our working set $\B^{M_2}$. Section \ref{seclink} is devoted to introduce our link and the implementation of the minimax structure.  In Section \ref{secinv}, we present our abstract approach to negative gradient flow on a constraint. and also establish the key Proposition \ref{propinva}. In Section \ref{secproof}, we conclude the proof of Theorem \ref{thmsc}. Section \ref{secproof2} is devoted to the proofs of Theorem \ref{thmsc2} and of Theorem \ref{thmsc3}.

\vskip0.1in
\noindent{\bf Notations}
\begin{itemize}
\item $\B_\rho^\mu = \bigl\{u \in H^1_\mu(\G): \int_\G|u'|^2dx < \rho\bigr\}.$
\item $\U_\rho^\mu = \bigl\{u \in H^1_\mu(\G): \int_\G|u'|^2dx = \rho\bigr\}.$
\item Let $0=\la_1(\G) < \la_2(\G) \leq \la_3(\G) \leq \cdots $ be the eigenvalues of $-\frac{d^2}{dx^2}$ on the compact metric graph $\G$ with Kirchhoff condition at the vertices, and let $\phi_1, \phi_2, \phi_3, \cdots$ be the corresponding eigenfunctions with $\int_\G|\phi_i|^2dx = 1$ for all $i \in \N$. Then we set
\begin{align*}
	\s_1 = span\{\phi_1\} \cap H^1_\mu(\G) = \{\pm \sqrt{\mu/l}\}, \quad \s_1^\perp = (span\{\phi_1\})^\perp \cap H^1_\mu(\G).
\end{align*}
\item $\K = \bigl\{u \in H^1_\mu(\G): (E|_{H^1_\mu(\G)})'(u,\G) = 0\bigl\}.$
\item  $\K[\alpha,\beta] = \bigl\{u \in H^1_\mu(\G): (E|_{H^1_\mu(\G)})'(u,\G) = 0, \alpha \leq E(u) \leq \beta\bigr\}.$
\item $\pm \p = \bigl\{u \in H^1(\G): \pm u \geq 0\bigl\}.$
\item $(\pm \p)_\nu = \bigl\{u \in H^1(\G): \inf_{w \in \pm \p}\|u -w\|_{H^1(\G)} \le \nu\bigr\}$.
\item  $D^*(\nu) =  \p_\nu \cup(-\p)_\nu$.
\item  $S^*(\nu) = H^1(\G) \backslash D^*(\nu)$.
\item  For $U \subset H^1_\mu(\G)$ and $\delta >0$ let $U_{\delta}$ be 
	\begin{align}
		U_{\delta} := \bigl\{v \in H^1_\mu(\G): dist(v, U) < \delta\bigr\}.
	\end{align}
\item  For $M \in \R$, $\B^M := \{u \in B_{\rho^*}^{\mu} : E(u, \G) < M\}$ with $\rho^*  := \min\{b,b\mu^{-\frac{p+2}{p-6}}\} \text{ where } b = \left( \frac2{(p-2)K}\right)^{\frac4{p-6}},$ is introduced in Proposition \ref{propinf}.
\item $\B^{M_2}$ is defined in \eqref{ajout2}.
\end{itemize}

\section{Preliminary results} \label{secpre}

At the start of this section, we recall a result in \cite{Dov}, which shows that any bounded (PS) sequence of $E(\cdot,\G)$ constrained on $H^1_\mu(\G)$ has a convergent subsequence in $H^1_\mu(\G)$.

\begin{lemma} \cite[Proposition 3.1]{Dov} \label{lemcompact}
	Assume that $\G$ is a compact graph and $\{u_n\} \subset H^1_\mu(\G)$ is a bounded
	(PS) sequence of $E(\cdot,\G)$ constrained on $H^1_\mu(\G)$. Then there exists $u \in H^1_\mu(\G)$ such that, up to a subsequence, $u_n \to u$ strongly in $H^1(\G)$.
\end{lemma}

Since (PS) sequences are, in general, unbounded, we introduce the following bounded subsets of the mass constraint for $\rho > 0$:
\begin{align*}
	\B_\rho^\mu = \bigl\{u \in H^1_\mu(\G): \int_\G|u'|^2dx < \rho\bigr\}, \quad \U_\rho^\mu = \bigl\{u \in H^1_\mu(\G): \int_\G|u'|^2dx = \rho\bigr\}.
\end{align*}
Similar bounded subsets were introduced for supercritical NLS with a partial confinement in \cite{BBJV} and in the bounded domain case in \cite{PV} to search for local minimizers.  

As a corollary of Lemma \ref{lemcompact}, any (PS) sequence of $E(\cdot,\G)$ in $\B_\rho^\mu$ has a convergent subsequence.

\begin{corollary} \label{corpscon}
	Assume that $\G$ is a compact graph and $\{u_n\} \subset \B_\rho^\mu$ is a (PS) sequence of $E(\cdot,\G)$ constrained on $H^1_\mu(\G)$. Then there exists $u \in H^1_\mu(\G)$ such that, up to a subsequence, $u_n \to u$ strongly in $H^1_\mu(\G)$. Furthermore, if 
	$\limsup_{n \to \infty}E(u_n,\G) < \inf_{\U_\rho^\mu}E(\cdot,\G)$, then $u \in \B_\rho^\mu$.
\end{corollary}

\begin{lemma} \label{lempregn}
	For any $u \in H^1(\G)$, there is a constant $K > 0$ depending on $\G$ and on $p$ such that
	\begin{align} \label{eqg-n1}
		\int_\G|u-\bar{u}|^pdx \leq K\left( \int_\G|u'|^2dx\right)^{\frac{p-2}4}\left( \int_\G|u-\bar{u}|^2dx\right)^{\frac{p+2}4},
	\end{align}
where 
$$
\bar{u} = |\G|^{-1}\int_\G udx.
$$
\end{lemma}

\begin{proof}
	 By the classical Gagliardo-Nirenberg inequality, we have
	\begin{align*}
		\int_\G|u-\bar u|^pdx \leq C\left( \int_\G(|(u-\bar u)'|^2 + |u-\bar u|^2)dx\right)^{\frac{p-2}4}\left( \int_\G|u-\bar{u}|^2dx\right)^{\frac{p+2}4},
	\end{align*}
    where $C > 0$ is a positive constant depending on $\G$ and on $p$. Since $\int_\G(u-\bar u)dx = 0$, we have
    \begin{align*}
    	\la_2(\G)\int_\G|u-\bar u|^2dx \leq \int_\G|(u-\bar u)'|^2dx = \int_\G|u'|^2dx.
    \end{align*}
    Hence,
    \begin{align*}
    	\int_\G|u-\bar u|^pdx \leq C\left( 1 + \frac1{\la_2(\G)} \right)^{\frac{p-2}4} \left( \int_\G|u'|^2dx\right)^{\frac{p-2}4}\left( \int_\G|u-\bar{u}|^2dx\right)^{\frac{p+2}4}.
    \end{align*}
\end{proof}

\begin{lemma} \label{lemgn}
	For any $u \in H_\mu^1(\G)$, there is a constant $K > 0$ depending on $\G$ and on $p$ such that
	\begin{align}
		\int_\G|u|^pdx & \leq pK\left( \int_\G|u'|^2dx\right)^{\frac{p-2}4}\left( \int_\G|u-\bar{u}|^2dx\right)^{\frac{p+2}4} + p\ell^{\frac{2-p}2}\mu^\frac p2 \label{eqgn1} \\
		& \leq pK\mu^{\frac{p+2}4}\left( \int_\G|u'|^2dx\right)^{\frac{p-2}4} + p\ell^{\frac{2-p}2}\mu^\frac p2, \label{ajout1}
	\end{align}
where $\ell$ denotes the total length of $\G$ and
$$
\bar{u} = |\G|^{-1}\int_\G udx.
$$
\end{lemma}

\begin{proof}
	By Lemma \ref{lempregn}, we know
	\begin{align*}
		\int_\G|u-\bar{u}|^pdx \leq K\left( \int_\G|u'|^2dx\right)^{\frac{p-2}4}\left( \int_\G|u-\bar{u}|^2dx\right)^{\frac{p+2}4}
	\end{align*}
where $K$ is a positive constant depending on $\G$ and on $p$. Then, we obtain
\begin{align} \label{equp}
	\int_\G|u|^pdx & = \int_\G|u-\bar{u} + \bar{u}|^pdx \nonumber \\
	& \leq p\int_\G|u-\bar{u}|^pdx + p|\bar{u}|^p|\G| \nonumber \\
	& \leq pK\left( \int_\G|u'|^2dx\right)^{\frac{p-2}4}\left( \int_\G|u-\bar{u}|^2dx\right)^{\frac{p+2}4} + p|\bar{u}|^p|\G|. 
\end{align}
Note that
\begin{align*}
	\int_\G(u-\bar{u})\bar{u}dx = 0.
\end{align*}
Then we have
\begin{align} \label{equ2}
	\int_\G|\bar{u}|^2dx + \int_\G|u-\bar{u}|^2dx = \int_\G|u|^2dx = \mu,
\end{align}
implying that $|\bar{u}| \leq (\mu/\ell)^{\frac12}$. This together with \eqref{equp} shows that
\begin{align*}
	\int_\G|u|^pdx \leq pK\left( \int_\G|u'|^2dx\right)^{\frac{p-2}4}\left( \int_\G|u-\bar{u}|^2dx\right)^{\frac{p+2}4} + p\ell^{\frac{2-p}2}\mu^\frac p2.
\end{align*}
Namely that \eqref{eqgn1} holds. Now, using \eqref{equ2} again, we have $\int_\G|u-\bar{u}|^2dx \leq \mu$ and thus \eqref{ajout1} follows from \eqref{eqgn1}. This completes the proof.
%\begin{align}
%	\int_\G|u|^pdx \leq pK\mu^{\frac{p+2}4}\left( \int_\G|u'|^2dx\right)^{\frac{p-2}4} + p\ell^{\frac{2-p}2}\mu^\frac p2,
%\end{align}
%completing the proof.

\end{proof}

\begin{remark}
	Since $\int_\G(u-\bar{u})dx = 0$, 
	$$
	\lambda_2(\G)\int_\G|u-\bar{u}|^2dx \leq \int_\G|u'|^2dx.
	$$
	%$$
	%\lambda_2(\G)\int_\G|u-\bar{u}|^2dx \leq \int_\G|u'|^2dx,
	%$$
	Then using \eqref{eqgn1} we see that
	\begin{align*}
		\int_\G|u|^pdx \leq p\bar{K}\left( \int_\G|u'|^2dx\right)^{\frac p2} + p\ell^{\frac{2-p}2}\mu^\frac p2, \quad \forall u \in H_\mu^1(\G),
	\end{align*}
where $\bar{K} = \lambda_2(\G)^{-\frac{p+2}4}K$ depends on $\G$ and on $p$.
\end{remark}

\begin{proposition} \label{propinf}
	Given $\mu >0$ and recalling that $p>6$, we define
	\begin{align*}
		& g(\rho) = \frac12\rho - K\mu^{\frac{p+2}4} \rho^{\frac{p-2}4}, \quad \rho > 0,  \\
		& \rho^* = \rho^*(\mu,p,\G) := \min\{b,b\mu^{-\frac{p+2}{p-6}}\} \text{ where } b = \left( \frac2{(p-2)K}\right)^{\frac4{p-6}}, \\ 
		& M_1 = M_1(\mu,p,\G) := g(\rho^*), \\
		& M_2 = M_2(\mu,p,\G) := M_1 - \ell^{\frac{2-p}2}\mu^{\frac p2},
	\end{align*}
where $K$ is the positive constant given by Lemma \ref{lemgn}. Note that $g'(\rho) > 0$ in $(0,b\mu^{-\frac{p+2}{p-6}})$ and $g'(\rho) < 0$ when $\rho > b\mu^{-\frac{p+2}{p-6}}$. We have 
\begin{align} \label{eqM1}
	M_1 \geq \frac{p-6}{2(p-2)}\rho^*,
\end{align}
and the ``=" holds if and only if $\mu \geq 1$. Furthermore, for any $\mu > 0$, we have
\begin{align}
	\inf_{\U^\mu_{\rho^*}} E(\cdot,\G) \geq M_2. \label{lowerbound}
\end{align}
\end{proposition}

\begin{proof}
	A direct computation yields
	\begin{align*}
		g'(\rho) = \frac12 - \frac{p-2}4K\mu^{\frac{p+2}4} \rho^{\frac{p-6}4}.
	\end{align*}
    Let $\bar{\rho} := b\mu^{-\frac{p+2}{p-6}}$. Obviously, $g'(\rho) > 0$ in $(0,\bar{\rho})$ and $g'(\rho) < 0$ when $\rho > \bar{\rho}$.
    
	If $\mu \geq 1$, then $\bar{\rho} = b\mu^{-\frac{p+2}{p-6}} \leq b$ and thus $\rho^* = \bar{\rho}$. Hence,
	\begin{align} \label{eqM1-1}
		 M_1 = g(\bar{\rho}) = \frac{p-6}{2(p-2)}b\mu^{-\frac{p+2}{p-6}} =  \frac{p-6}{2(p-2)}\rho^*.
	\end{align}
    If $0 < \mu < 1$, then $\bar{\rho} > b$ and thus $\rho^* = b < \bar{\rho}$. Hence,
    \begin{align}\label{eqM1-2}
    	M_1 = g(b) = \frac12b - K\mu^{\frac{p+2}4}b^{\frac{p-2}4} > \frac12b - Kb^{\frac{p-2}4} = \frac{p-6}{2(p-2)}b = \frac{p-6}{2(p-2)}\rho^*.
    \end{align}
    By \eqref{eqM1-1} and \eqref{eqM1-2}, we get \eqref{eqM1} and show that the "=" in \eqref{eqM1} holds if and only if $\mu \geq 1$.
    
    %It can be verified that $g(\rho)$ is strictly increasing when $0 < \rho < \bar{\rho}$, and obtain its maximum value at $\rho = \bar{\rho}$.
	Furthermore, by Lemma \ref{lemgn}, we have
	\begin{align*}
		\inf_{\U^\mu_{\rho^*}} E(\cdot,\G) \geq \frac12\rho^* - K\mu^{\frac{p+2}4}\left( \rho^*\right)^{\frac{p-2}4} - \ell^{\frac{2-p}2}\mu^\frac p2 = g(\rho^*) - \ell^{\frac{2-p}2}\mu^\frac p2 = M_2.
	\end{align*}
	Namely \eqref{lowerbound} holds.
\end{proof}

\begin{remark}
	In the definition of $\rho^*$, we made sure that $\rho^* \leq b$ in order to have a uniform bound with respect to $\mu$ as $\mu \to 0$. The boundedness of $\rho^*$ will be used in the proof of Lemma \ref{lemG(u)}.
\end{remark}

\section{Link and the min-max structure} \label{seclink}

Let $0=\la_1(\G) < \la_2(\G) \leq \la_3(\G) \leq \cdots $ be the eigenvalues of $-\frac{d^2}{dx^2}$ on the compact metric graph $\G$ with Kirchhoff condition at the vertices, and let $\phi_1, \phi_2, \phi_3, \cdots$ be the corresponding eigenfunctions with $\int_\G|\phi_i|^2dx = 1$ for all $i \in \N$. Then we set
\begin{align*}
	\s_1 = span\{\phi_1\} \cap H^1_\mu(\G) = \{\pm \sqrt{\mu/l}\}, \quad \s_1^\perp = (span\{\phi_1\})^\perp \cap H^1_\mu(\G).
\end{align*}
Note that since $\phi_1 = \frac{1}{\sqrt{l}},$ any $u \in \s_1^\perp$ is sign-changing.
	
	A compact subset $A$ of $H^1_\mu(\G)$ with $\s_1 \subset A$ and $A \cap \mathbb{S}_{1}^\perp \not= \emptyset$ is said to be \textit{linked} to $\mathbb{S}_{1}^\perp$ if, for any continuous mapping $h \in C(A,H^1_\mu(\G))$ satisfying $h|_{\s_1} = \textbf{id}$, there holds $h(A) \cap \mathbb{S}_{1}^\perp \not= \emptyset$. Recall that $\rho^*$ and $M_2$ were defined in Proposition \ref{propinf}. We introduce the following set which will be used to define a min-max structure:
\begin{align} \label{eqdefl1}
	\cl_1 = \cl_1(\mu) := \bigl\{A \subset \B_{\rho^*}^\mu: \sup_A E(\cdot,\G) < M_2, \ \s_1 \subset A, \ A \text{ is compact and is linked to } \s_1^\perp\bigr\}.
\end{align}

\begin{lemma} \label{lemlnotemp}
    Let $\hat{\mu} = \hat{\mu}(\G,p) > 0$ be the positive solution of
    $$
    \la_2(\G)\mu + 2\ell^{\frac{2-p}2}\mu^{\frac p2} = \frac{p-6}{p-2}\left( \frac2{(p-2)K}\right)^{\frac4{p-6}},
    $$
    and let $\bar{\mu} = \bar{\mu}(\G,p) > 0$ be the positive solution of
    $$
    \la_2(\G)\mu^{\frac{2(p-2)}{p-6}} + 2\ell^{\frac{2-p}2}\mu^{\frac{(p-2)^2}{2(p-6)}} = \frac{p-6}{p-2}\left( \frac2{(p-2)K}\right)^{\frac4{p-6}}.
    $$
	Then $\cl_1$ is not empty for $0 < \mu < \min\{\hat{\mu},\bar{\mu}\}$.
\end{lemma}

\begin{proof}
	Define
	$$
	Q_2:= \left\{t_1\sqrt{\mu}\phi_1+t_2\sqrt{\mu}\phi_2: t^2_1 + t^2_2 = 1, t_2 \geq 0\right\} \subset H^1_\mu(\G).
	$$
	It is clear that $\s_1 = \{\pm \sqrt{\mu/l}\} \subset Q_2$. By the definitions of $\hat{\mu}$ and $\bar{\mu}$, we know
	$$
	\hat{\mu} > 1 \Leftrightarrow \bar{\mu} > 1,  \quad 0 < \hat{\mu} < 1 \Leftrightarrow 0< \bar{\mu} < 1.
	$$
	If $0 < \mu < \min\{\hat{\mu}, \bar{\mu}\}$ and $\mu \leq 1$,  we have
	\begin{align} \label{eqm1}
		\la_2(\G)\mu + 2\ell^{\frac{2-p}2}\mu^{\frac p2} < \frac{p-6}{p-2}\left( \frac2{(p-2)K}\right)^{\frac4{p-6}}.
	\end{align}
    Since $\mu \leq 1$, $\rho^* = b$ where $b = \left( \frac2{(p-2)K}\right)^{\frac4{p-6}}$. By Proposition \ref{propinf} and \eqref{eqm1}, direct computations yield
    \begin{align} \label{eqQ2-1}
    	\sup_{Q_2} E(\cdot,\G) < \frac12\la_2(\G)\mu < \frac{p-6}{2(p-2)}b - \ell^{\frac{2-p}2}\mu^{\frac p2} \leq M_1 - \ell^{\frac{2-p}2}\mu^{\frac p2} = M_2.
    \end{align}
    If $1 < \mu < \min\{\hat{\mu}, \bar{\mu}\}$, we have
    \begin{align} \label{eqm1-1}
    	\la_2(\G)\mu^{\frac{2(p-2)}{p-6}} + 2\ell^{\frac{2-p}2}\mu^{\frac{(p-2)^2}{2(p-6)}} < \frac{p-6}{p-2}\left( \frac2{(p-2)K}\right)^{\frac4{p-6}}.
    \end{align}
    Since $\mu > 1$, $\rho^* = b\mu^{-\frac{p+2}{p-6}}$. By Proposition \ref{propinf} and \eqref{eqm1-1}, direct computations yield
    \begin{align} \label{eqQ2-2}
    	\sup_{Q_2}E(\cdot,\G) < \frac12\la_2(\G)\mu < \frac{p-6}{2(p-2)}b\mu^{-\frac{p+2}{p-6}} - \ell^{\frac{2-p}2}\mu^{\frac p2} = M_1 - \ell^{\frac{2-p}2}\mu^{\frac p2} = M_2.
    \end{align}
    By \eqref{eqQ2-1} and \eqref{eqQ2-2} we deduce that $\sup_{Q_2} E(\cdot,\G) < M_2$ when $0 < \mu < \min\{\hat{\mu},\bar{\mu}\}$.
    
    Next, we show $Q_2 \subset \B_{\rho^*}^\mu$ if $0 < \mu < \min\{\hat{\mu},\bar{\mu}\}$. When $0 < \mu < \min\{\hat{\mu}, \bar{\mu}\}$ and $\mu \leq 1$, \eqref{eqm1} implies
    \begin{align} \label{eqm2}
    	\la_2(\G)\mu < \frac{p-6}{p-2}b < b = \rho^*.
    \end{align} 
    When $1 < \mu < \min\{\hat{\mu}, \bar{\mu}\}$, \eqref{eqm1-1} implies
    \begin{align} \label{eqm2-1}
    	\la_2(\G)\mu < \frac{p-6}{p-2}b\mu^{-\frac{p+2}{p-6}} < b\mu^{-\frac{p+2}{p-6}} = \rho^*.
    \end{align}
    By \eqref{eqm2} and \eqref{eqm2-1} we obtain $\int_\G|u'|^2 dx \leq \la_2(\G)\mu < \rho^*$ for any $u \in Q_2$, implying that $Q_2 \subset \B_{\rho^*}^\mu$ for $0 < \mu < \min\{\hat{\mu},\bar{\mu}\}$.

    It remains to prove that $Q_2$ is linked to $\s_1^\perp$. 
		
		For any continuous mapping $h \in C(Q_2,H^1_\mu(\G))$, we have to show that $h(Q_2) \cap \mathbb{S}_1^\perp \not= \emptyset$. Since $Q_2$ is one-dimensional and  necessarily $\int_{\G} h(- \sqrt{\mu} \phi_1) \phi_1 <0$ and $\int_{\G} h( \sqrt{\mu} \phi_1) \phi_1 >0$, the intermediate value theorem gives directly the result. However, having in view the proof of Theorem \ref{thmsc2}, we present here a homotopy argument (see, e.g., Example 2 in \cite[Chapter II, Section 1]{Cha}). Let $P : H^1(\G) \to span\{\phi_1\}$ be the orthogonal projection. Note that $Q_2 \subset span\{\phi_1,\phi_2\}$. Hence, the topological degrees $\deg(P \circ h,Q_2,0)$ and $\deg(P|_{Q_2},Q_2,0)$ are well-defined. Let
    \begin{align*}
	    g(t,u) := (1-t)P(h(u)) + t P|_{Q_2} u, \quad \forall t \in [0,1], \quad u \in Q_2.
    \end{align*}
    For any $u \in \partial Q_2 = \s_1$ and $t \in [0,1]$, we have $g(t,u) = P u = u \neq 0$. Hence, we have
    \begin{align*}
	    \deg(P \circ h,Q_2,0) = \deg(P|_{Q_2},Q_2,0) = 1,
    \end{align*}
    and so, there exists a $u_0 \in Q_2$ such that $P(h(u_0)) = 0$, i.e., $h(u_0) \in \mathbb{S}_1^\perp$. This yields that $h(Q_2) \cap \mathbb{S}_1^\perp \not= \emptyset$ and completes the proof.
\end{proof}

To find sign-changing critical points, we introduce the following sets
\begin{align*}
   	& \pm \p = \bigl\{u \in H^1(\G): \pm u \geq 0\bigl\},\\
	%& \ \K = \bigl\{u \in H^1_\mu(\G): (E|_{H^1_\mu(\G)})'(u,\G) = 0\bigl\},\\
	%& \ \K[\alpha,\beta] = \bigl\{u \in H^1_\mu(\G): (E|_{H^1_\mu(\G)})'(u,\G) = 0, \alpha \leq E(u) \leq \beta\bigr\}. \\
	&  (\pm \p)_\nu = \bigl\{u \in H^1(\G): \inf_{w \in \pm \p}\|u -w\|_{H^1(\G)} \leq \nu\bigr\}, \quad \text{for some } \nu > 0,  \\
	& \ D^*(\nu) = \p_\nu \cup (-\p)_\nu, \quad S^*(\nu) = H^1(\G) \backslash D^*(\nu).
\end{align*}

%Let $D^* =  \p \cup(-\p)$ and $S^* = H^1(\G) \backslash D^*$. 
Using $\cl_1$ and $S^*(\nu)$ we can define a min-max value by
\begin{align*}
	c_1 := \inf_{A \in \cl_1}\sup_{u \in A\cap S^*(\nu)}E(u,\G).
\end{align*}
Let 
\begin{equation}\label{ajout2}
\B^{M_2} := \bigl\{u \in \B_{\rho^*}^\mu: E(u,\G) < M_2\bigl\},   
\end{equation}
where $M_2$ is defined in Proposition \ref{propinf}
and
\begin{align*}
	\bar{c}_1 := \inf_{u\in\s_1^\perp \cap \B^{M_2}}E(u,\G). 
\end{align*}

	We need the following Lemma \ref{lemfarawayfrom} to show that $c_1$ is well defined for small $\nu > 0$.
	
	\begin{lemma} \label{lemfarawayfrom}
		Let $\delta := \min\{dist(\mathbb{S}_{1}^\perp \cap \B_{\rho^*}^\mu, \p), dist(\mathbb{S}_{1}^\perp \cap \B_{\rho^*}^\mu, -\p)\} $. Then $\delta > 0$.
	\end{lemma}
	
	\begin{proof}
		For any $u \in \mathbb{S}_{1}^\perp \cap \B_{\rho^*}^\mu$, the function $-u$ also belongs to $\mathbb{S}_{1}^\perp \cap \B_{\rho^*}^\mu$. Thus, it is readily seen that $dist(\mathbb{S}_{1}^\perp \cap \B_{\rho^*}^\mu, \p) = dist(\mathbb{S}_{1}^\perp \cap \B_{\rho^*}^\mu, -\p)$. By contradiction, we assume that
		$$
		dist(\mathbb{S}_{1}^\perp \cap \B_{\rho^*}^\mu, \p) = 0.
		$$
		Then there exist $\{u_n\} \subset \mathbb{S}_{1}^\perp \cap \B_{\rho^*}^\mu$, $\{w_n\} \subset \p$ such that $\|u_n - w_n\|_{H^1(\G)} \to 0$ as $n \to \infty$. Note that $\B_{\rho^*}^\mu$ is bounded in $H^1(\G)$. Thus $\{u_n\}$, and also $\{w_n\}$, are bounded in $H^1(\G)$. Up to a subsequence as $n \to \infty$, we can assume that $u_n \rightharpoonup u^*$  weakly in $H^1(\G)$, $u_n \to u^*$ strongly in $L^2(\G)$, $w_n \rightharpoonup w^* \in \p$ weakly in $H^1(\G)$, and $w_n \to w^*$ strongly in $L^2(\G)$. Then we observe that $u^* \in \mathbb{S}_{1}^\perp$ with $u^* \in H_{\mu}^1(\G)$
		%$\|u^*\|_{L^2(\G)}^2 = \mu$ 
		which implies that $u^* \not\equiv 0$ is sign-changing. However, we know $u^* \equiv w^*$ from $\|u_n - w_n\|_{H^1(\G)} \to 0$, which contradicts that $w^* \in \p$. 
	\end{proof}

\begin{lemma} \label{lemwelldef}
	For $0 < \mu < \min\{\hat{\mu},\bar{\mu}\}$ and $0 < \nu < \delta$ where $\delta$ is given in Lemma \ref{lemfarawayfrom}, both values $c_1$ and $\bar{c}_1$ are well defined and $- \infty <\bar{c}_1 \leq c_1 < M_2$.
\end{lemma}

\begin{proof}
Let $0 < \mu < \min\{\hat{\mu},\bar{\mu}\}$. From the proof of Lemma \ref{lemlnotemp}, we know that $\sqrt{\mu}\phi_2 \in \B^{M_2}$ and thus, since $\sqrt{\mu}\phi_2 \in \s_1^\perp$, the set $\s_1^\perp \cap \B^{M_2}$ is not empty. This shows that $\bar{c}_1$ is well defined. Obviously, $\bar{c}_1 < M_2 < \infty$ and also, $ - \infty <\bar{c}_1 $, since Lemma \ref{lemgn} guarantees that $E(\cdot, \G)$ is bounded on bounded sets of $H^1(\G)$.
	
Lemma \ref{lemlnotemp} yields that $\cl_1$ is not empty for $0 < \mu < \min\{\hat{\mu},\bar{\mu}\}$. Furthermore, by the definitions of $\cl_1$ and the link, for any $A \subset \cl_1$ there exists $u_0 \in A$ such that $u_0 \in \s_1^\perp \cap \B^{M_2}$. For $0 < \nu < \delta$, by Lemma \ref{lemfarawayfrom} we have $u_0 \in S^*(\nu)$. Thus $A \cap S^*(\nu)$ is not empty for any $A \in \cl_1$ and we obtain
	\begin{align*}
		\bar{c}_1 = \inf_{u\in\s_1^\perp \cap \B^{M_2}}E(u,\G) \leq \inf_{A \in \cl_1}\sup_{u \in A\cap S^*(\nu)}E(u,\G) = c_1.
	\end{align*}
	Clearly also $c_1$ is finite. Indeed taking an arbitrary $A \in \cl_1$ we have that
	$$ \sup_{u \in A \cap S^*}E(u, \G) < \infty$$
	by the compactness of $A$. The proof is now complete.
	 %To prove \textcolor{red}{that} $c_1$ is well defined, it suffices to show that $c_1 < \infty$. Indeed, in the proof of Lemma \ref{lemlnotemp} we know $Q_2 \in \cl_1$, implying that 
	% \begin{align*}
	% 	c_1 = \inf_{A \in \cl_1}\sup_{u \in A\cap S^*}E(u,\G) \leq \sup_{u \in Q_2\cap S^*}E(u,\G) < M_2 < \infty.
	% \end{align*}
  %   The proof is complete.
\end{proof}

Setting 
$$\underline{c}_1 := \sup_{u \in \s_1}E(u,\G) = -\frac1p\ell^{\frac{2-p}2}\mu^{\frac p2},$$
we will also need the following lower bound.

\begin{lemma} \label{lemvaluesepe}
	Let $0 < \mu < \min\{\hat{\mu},\bar{\mu}, \mu^*\}$ where $\mu^* = \mu^*(\G,p) > 0$ is the positive solution of
	$$
	2\left( \frac{p-1}p\ell^{\frac{2-p}2} + K\la_2(\G)^{\frac{p-2}4}\right) \mu^{\frac {p-2}2} = \la_2(\G).
	$$
	Then $\bar{c}_1 > \underline{c}_1$.
\end{lemma}

\begin{proof}
We recall that the function $g(\rho) = \frac12\rho - K\mu^{\frac{p+2}4} \rho^{\frac{p-2}4}$, which was defined in Proposition \ref{propinf}, is strictly increasing in $(0,\rho^*)$. Noting that $\int_\G|u'|^2dx \geq \la_2(\G)\mu$, for any $u \in \s_1^\perp$, it follows
\begin{align*} 
		\inf_{u\in\s_1^\perp \cap \B^{M_2}}E(u,\G)
		& \geq \inf_{u\in\s_1^\perp \cap \B^{M_2}}\left( \frac12\int_\G|u'|^2dx - K\mu^{\frac{p+2}4}\left( \int_\G|u'|^2dx\right)^{\frac{p-2}4} - \ell^{\frac{2-p}2}\mu^\frac p2 \right) \nonumber \\
		& \geq g(\la_2(\G)\mu)- \ell^{\frac{2-p}2}\mu^\frac p2.
	\end{align*}
	From the definition of $g(\rho)$ and the one of $\mu^*$, we obtain, for $0 < \mu < \min\{\hat{\mu},\bar{\mu}, u^*\}$,  
    $$
    \bar{c}_1 = \inf_{u\in\s_1^\perp \cap \B^{M_2}}E(u,\G) > -\frac1p\ell^{\frac{2-p}2}\mu^{\frac p2} = \underline{c}_1.
    $$
    The proof is complete.
\end{proof}

In Section \ref{secproof} below, for sufficiently small $\nu > 0$, we will prove that there exists a sign-changing critical point (contained in $\B^{M_2} \cap S^*(\nu)$) of $E(\cdot,\G)$ in $H^1_\mu(\G)$ at the level $c_1$. 

\section{Invariant sets under the negative gradient flow} \label{secinv}

Let $\eta(t,\cdot)$ be a decreasing flow on $H^1_\mu(\G)$ corresponding to the gradient $\nabla E(\cdot,\G)$ ($\nabla E(\cdot, \G)$ means the gradient of $E(\cdot, \G)$ as a constrained functional on $H^1_\mu(\G)$) defined by
\begin{align} \label{eqflow}
	\left\{
	\begin{aligned}
		& \frac{\partial}{\partial t}\eta(t,u) = -h(\eta(t,u))y(\eta(t,u))\|\nabla E(\eta(t,u),\G)\|_{H^1(\G)}^{-2}\nabla E(\eta(t,u),\G), \\
		& \eta(0,u) = u,
	\end{aligned}
	\right.
\end{align}
where $y(u): H^1_\mu(\G) \to [0,1]$ is a locally Lipschitz continuous map such that
\begin{align*}
	y(u) :=
	\left\{
	\begin{aligned}
		& 1 \quad \text{for all } u \in H^1_\mu(\G) \backslash(\mathcal{K}[c_1-\bar{\epsilon},c_1+\bar{\epsilon}])_{\delta_1/2}, \\
		& 0 \quad \text{for all } u \in (\mathcal{K}[c_1-\bar{\epsilon},c_1+\bar{\epsilon}])_{\delta_1/3},
	\end{aligned}
	\right.
\end{align*}
and $$h(u):= \frac{dist(u,U_1)}{dist(u,U_2)+dist(u,U_1)}$$ with
\begin{align*}
	U_1 := \bigl\{u \in H^1_\mu(\G): |E(u,\G) - c_1| \geq 3 \epsilon_1\bigr\}, \quad U_2 := \bigl\{u \in H^1_\mu(\G): |E(u,\G)) - c_1| \leq 2 \epsilon_1\bigr\}. 
\end{align*}

Here $\delta_1 > 0$, $0< 3\epsilon_1 \leq \bar{\epsilon}$ will be chosen in the proof of Theorem \ref{thmsc}, and we recall that $\K[\alpha,\beta] = \bigl\{u \in H^1_\mu(\G): (E|_{H^1_\mu(\G)})'(u,\G) = 0, \alpha \leq E(u) \leq \beta\bigr\}.$

\begin{lemma} \label{lemgloext}
	Let $\eta(t,\cdot)$ be given by \eqref{eqflow}. Then for any $u \in \B^{M_2}$, with $\B^{M_2}$ defined in \eqref{ajout2},  $\eta(t,u)$ exists for $t \in [0,\infty)$, and, $\eta(t,u) \in \B^{M_2}$ for all $t > 0$.
\end{lemma}

\begin{remark}
	 It is well known that $\eta(t,u)$ globally exists for any $u \in H^1_\mu(\G)$ if $E(\cdot,\G)$ satisfies the (PS) condition on $H^1_\mu(\G)$. However, the global existence result in the above lemma is not direct since the (PS) condition is lost in our setting. We shall make use of Lemma \ref{lemcompact} to prove it.
\end{remark}

\begin{proof}[Proof of Lemma \ref{lemgloext}]
From Corollary \ref{corpscon}, we deduce the existence of a positive constant $C > 0$ such that $\|\nabla E(u,\G)\|_{H^1(\G)} \geq 1/C$ for any $u \in  E(\cdot,\G)^{-1}[c_1- \bar{\epsilon}, c_1 + \bar{\epsilon}] \cap \B^{M_2} \backslash (\mathcal{K}[c_1-\bar{\epsilon},c_1+\bar{\epsilon}])_{\delta_1/3}$, and so 
\begin{equation}\label{global}
\|h(u)y(u)\|\nabla E(u,\G)\|_{H^1(\G)}^{-2}\nabla E(u,\G)\|_{H^1(\G)} = h(u)y(u)\|\nabla E(u,\G)\|_{H^1(\G)}^{-1} \leq C, \quad \text{for } u \in \B^{M_2}.
\end{equation}
It is standard that $\eta(t,u)$ locally exists for any $u \in H^1_\mu(\G)$.
%see for example, \cite{Schw}. 
Let us prove that $\eta(t,u) \in \B^{M_2}$ for any $u \in \B^{M_2}$ and all $t \in [0,T_u)$, where $T_u$ is the maximal time such that $\eta(t,u)$ exists. In fact,
	\begin{align*}
		E(\eta(t,u),\G) \leq E(u,\G) < M_2
	\end{align*}
	for any $u \in \B^{M_2}$ and $t \in [0,T_u)$. Hence, if $\eta(t_u,u) \notin \B^{M_2}$ for some $u \in \B^{M_2}$ and $0 < t_u < T_u$, then $\eta(t_u,u) \notin \B_{\rho^*}^\mu$. By continuity, there exists $t_u' \in (0,t_u]$ such that $\eta(t_u',u) \in \U_{\rho^*}^\mu$. By Proposition \ref{propinf}, we have $E(\eta(t_u',u),\G) \geq M_2$, in contradiction with
	\begin{align*}
		E(\eta(t_u',u),\G) \leq E(u,\G) < M_2.
	\end{align*}
	Thus $\B^{M_2}$ is an invariant subset and, in view of \eqref{global}, the global existence follows.
    
   % Next, we argue by contradiction to prove $T_u = \infty$ for any $u \in \B^{M_2}$. Suppose that there exists $u \in \B^{M_2}$ such that $T_u < \infty$ and let
	%$$
	%W(u) := -h(u)y(u)\|\nabla E(u,\G)\|^{-2}\nabla E(u,\G).
	%$$
	%Since $T_u < \infty$, necessarily
	%\begin{align*}
	%	\limsup_{t \to T_u}\|W(\eta(t,u))\|_{H^1(\G)} \to \infty,
	%\end{align*}
	%and then 
	%\begin{align*}
	%	\liminf_{t \to T_u}\|\nabla E(u, \G)\|_{H^1(\G)} \to 0.
	%\end{align*}
	%Thus, if we consider a sequence $t_n \to T_{\mu}^-$, by Lemma \ref{lemcompact}, $\eta(t_n,u) \to \eta(T_u, u)$ contradicting the maximality of $T_u$. This contradiction completes the proof.
\end{proof}

To establish the existence of sign-changing critical points, we need to prove that, for sufficiently small $\nu > 0$, $D^*(\nu) \cap \B^{M_2}$ is invariant under the flow $\eta(t,\cdot)$, i.e., $\eta(t,u) \in D^*(\nu) \cap \B^{M_2}$ for any $u \in D^*(\nu) \cap \B^{M_2}$ and all $t > 0$. The main result in this section is as follows.

\begin{proposition} \label{propinva}
	Let $\eta(t,\cdot)$ be defined by \eqref{eqflow} and $0 < \mu \leq \tilde{\mu}$, where $\tilde{\mu} = \tilde{\mu}(\G,p)$ is the positive solution of 
	\begin{align*}
		pK\left( \frac2{(p-2)K}\right)^{\frac{p-2}{p-6}}\mu^{\frac{p-2}4} + p\ell^{\frac{2-p}2}\mu^{\frac{p-2}2} = 1.
	\end{align*}
	Then, there exists $\nu^* > 0$ such that when $0 < \nu < \nu^*$, we have $\eta(t,u) \in D^*(\nu) \cap \B^{M_2}$ for any $u \in D^*(\nu) \cap \B^{M_2}$ and all $t \in (0,\infty)$.
\end{proposition}

The following Proposition \ref{propbm} has been successfully used to prove the invariance property of sets such as $(\pm \p)_\nu$ under a negative gradient flow in unconstrained problems, see, e.g., \cite[Section 2]{SchZ1}, \cite[Section 2]{SchZ} and \cite[Lemma 2.3]{ZZ}. Readers can consult \cite[Theorem 4.1]{Dei} for its proof. Similar results can be found in \cite[Theorem 1]{Bre} for finite-dimensional Euclidean space cases and in \cite[Theorem 3]{Mar1} or \cite{Mar} for Banach space cases.
\medskip

\begin{proposition}[Br\'ezis-Martin] \label{propbm}
	Let $X$ be a Banach space, $B \subset X$, $u \in B$, $U_X$ be a neighborhood of $u$ in $X$ and $B \cap \overline{U_X}$ be closed. If $\widetilde V: \overline{U_X} \to X$ is a Lipschitz mapping and
	\begin{align} \label{eqd=0banachcase}
		\lim_{s \searrow 0}s^{-1}dist(w + s\widetilde V(w),B) = 0, \quad \forall w \in \overline{U_X} \cap B,
	\end{align}
	then there exist $r = r(u) > 0$ and $\tilde\eta(t,u)$ satisfying
	\begin{align*}
		\left\{
		\begin{aligned}
			& \frac{\partial}{\partial t}\tilde\eta(t,u) = \widetilde V(\tilde\eta(t,u)), \quad \forall t \in [0,r), \\
			& \tilde\eta(0,u) = u, \quad \tilde\eta(t,u) \in B.
		\end{aligned}
		\right.
	\end{align*}
\end{proposition}

Since there is a constraint in our setting, Proposition \ref{propbm} cannot be used directly to prove Proposition \ref{propinva}. First we shall extend it to a setting which will cover the mass constrained problem that we consider. Note that the setting and the results in Proposition \ref{propbmmanifold}, Lemma \ref{lemconvex}, and Corollary \ref{corbm} presented below may have an interest to treat other constrained problems.

Let $(E,\langle \cdot,\cdot\rangle)$ and $(H,( \cdot,\cdot))$ be two \emph{infinite-dimensional} Hilbert spaces and assume that
\begin{align*}
	E \hookrightarrow H \hookrightarrow E',
\end{align*}
with continuous injections. For simplicity, we assume that the continuous injection $E \hookrightarrow H$ has norm at most $1$ and identify $E$ with its image in $H$. We also introduce
\begin{align*}
	\left\{
	\begin{aligned}
		\|u\|^2 & = \langle u,u \rangle, \\
		|u|^2 \ & = (u,u), \\
	\end{aligned}
	\right.
	\quad \quad u \in E,
\end{align*}
and, for $\mu \in (0,\infty)$, we define
\begin{align*}
	S_\mu := \{u \in E: |u|^2 = \mu \}.
\end{align*}
For our application, it is plain that $E = H^1(\G)$ and $H = L^2(\G)$.

\begin{proposition} \label{propbmmanifold}
	Let $u \in B_\mu \subset S_\mu$, $U$ be a neighborhood of $u$ in $S_\mu$ and $B_\mu \cap \overline{U}$ closed in $S_\mu$. Let $V: \overline{U} \to E$ be a Lipschitz mapping on $\overline{U}$ with $V(w) \in T_wS_\mu$ for any $w \in \overline{U}$, where $T_wS_\mu$ is the tangent space of $S_\mu$ at $w$ defined as
	\begin{align*}
		T_wS_\mu := \bigl\{\phi \in E: (\phi,w) = 0\bigr\},
	\end{align*}
    which satisfies
    \begin{align} \label{eqd=0}
    	\lim_{s \searrow 0}s^{-1}dist(w + sV(w),B_\mu) = 0, \quad \forall w \in \overline{U} \cap B_\mu.
    \end{align}
    Then  there exist $r = r(u) > 0$ and $\eta(t,u)$ satisfying 
	\begin{align*} 
	%\label{eqeta}
		\left\{
			\begin{aligned}
				& \frac{\partial}{\partial t}\eta(t,u) = V(\eta(t,u)), \quad \forall t \in [0,r), \\
				& \eta(0,u) = u, \quad \eta(t,u) \in B_\mu.
			\end{aligned}
		\right.
	\end{align*}
\end{proposition}

\begin{proof}
	%If \eqref{eqeta} holds, then
	%\begin{align*}
	%	dist(u + sV(u),B_\mu) \leq \|u + sV(x) - \eta(s,u)\| = \|sV(x) + \eta(0,u) - \eta(s,u)\|
	%\end{align*}
    %for $s \in (0,r)$. Then we have
    %\begin{align*}
    %	\lim_{s \searrow 0}s^{-1}dist(u + sV(u),B_\mu) \leq \lim_{s \searrow 0}\bigl\|V(x) - \frac{\eta(s,u) - \eta(0,u)}{s}\bigr\| = 0,
    %\end{align*}
	%i.e., \eqref{eqd=0} holds true.
	
	%Next, we assume that \eqref{eqd=0} holds true and aim to prove \eqref{eqeta}. 
	Let us construct a suitable extension of $V$ on a neighbourhood of $u$ in $E$. Setting
	\begin{align*}
		B & := \bigl\{\theta w: \theta \in [1/2,3/2], w \in B_\mu\bigr\}, \\
		U_E & := \bigl\{\theta w: \theta \in (1/2,3/2), w \in U\bigr\},
	\end{align*}
    it is not difficult to check that $u \in B \subset E$, $U_E$ is a neighborhood of $u$ in $E$ and $B \cap \overline{U_E}$ is closed in $E$. For any $w \in \overline{U_E}$, we define
    \begin{align*}
    	\widetilde V(w) := \frac{|w|}{\sqrt \mu}V(\sqrt \mu w/|w|) \in E.
    \end{align*}
    Let us check that $\widetilde{V}$ is a Lipschitz mapping. Since $V$ is Lipschitz, for any $w, w' \in \overline{U}$,
    \begin{align*}
    	\|V(w')-V(w)\| \leq C\|w'-w\|.
    \end{align*}
    Here and in the sequel in this proof, $C > 0$ is a constant which may change line by line. Then, for any $\theta w, \theta'w' \in \overline{U_E}$ with $w, w' \in \overline{U}$,
    \begin{align*}
    	\|\widetilde{V}(\theta'w') - \widetilde{V}(\theta w)\| &
    	= \|\theta'V(w') - \theta V(w)\| \\
    	& \leq \theta'\|V(w')-V(w)\| + |\theta'-\theta|\|V(w)\| \\
    	& \leq C\theta'\|w'-w\| + |\theta'-\theta |\|V(w)\|. 
    \end{align*}
    Note that $\theta' = |\theta'w'|/\sqrt{\mu}$ and $\theta  = |\theta w|/\sqrt{\mu}$. We deduce
    \begin{align*}
    	|\theta'-\theta| \leq \mu^{-\frac12}|\theta'w'-\theta w|.
    \end{align*}
    Hence,
    \begin{align*}
    	\|\widetilde{V}(\theta'w') - \widetilde{V}(\theta w)\| & \leq C\|\theta'w'-\theta'w\| + |\theta'-\theta|\|V(w)\| \\
    	& \leq C\|\theta'w'-\theta w\| + \left( C\|w\|+\|V(w)\|\right) |\theta'-\theta| \\
    	& \leq C\|\theta'w'-\theta w\|+C\mu^{-\frac12}|\theta'w'-\theta w| \\
    	& \leq C\|\theta'w'-\theta w\|.
    \end{align*}
    Now, we apply Proposition \ref{propbm} with $X = E$ and $\widetilde V(u) = \mu^{-1/2}|u|V(\sqrt \mu u/|u|)$. We verify \eqref{eqd=0banachcase}. For any $\theta w \in \overline{U_E} \cap B$ with $w \in \overline{U} \cap B_\mu$, using \eqref{eqd=0} we get the existence of $v_s \in B_\mu$ such that 
    \begin{align*}
    	\lim_{s \searrow 0}s^{-1}\|w + sV(w)-v_s\| = 0.
    \end{align*}
    Hence,
    \begin{align*}
    	\lim_{s \searrow 0}s^{-1}dist(\theta w + s\widetilde V(\theta w),B) & \leq \lim_{s \searrow 0}s^{-1}\|\theta w + s\theta V(w) - \theta v_s\| = 0.
    \end{align*}
    Then, by Proposition \ref{propbm} we obtain the existence of $r = r(u) > 0$ and $\tilde\eta(t,u)$ such that 
    \begin{align*}
    	\left\{
    	\begin{aligned}
    		& \frac{\partial}{\partial t}\tilde\eta(t,u) = \widetilde V(\tilde\eta(t,u)), \quad \forall t \in [0,r), \\
    		& \tilde\eta(0,u) = u \in B_\mu, \quad \tilde\eta(t,u) \in B.
    	\end{aligned}
    	\right.
    \end{align*}
    Let $\eta(\cdot,u): [0,T) \to S_\mu$ be the local solution of 
    \begin{align*}
    	\left\{
    	\begin{aligned}
    		& \frac{\partial}{\partial t}\eta(t,u) = V(\eta(t,u)), \\ 
    		& \eta(0,u) = u.
    	\end{aligned}
        \right.	
    \end{align*}
   Clearly $\eta(\cdot,u)$ can be viewed as a function mapping $[0,T)$ to $E$ since $S_\mu \subset E$. It is not difficult to check that $\eta(\cdot,u): [0,T) \to E$ satisfies
    \begin{align} \label{eqetaonE}
    	\left\{
    	\begin{aligned}
    		& \frac{\partial}{\partial t}\eta(t,u) = \widetilde V(\eta(t,u)), \\ 
    		& \eta(0,u) = u.
    	\end{aligned}
    	\right.	
    \end{align}
    By the local existence and uniqueness of the solution for equation \eqref{eqetaonE} and by decreasing $r >0$ if necessary, we know that $\eta(t,u) \equiv \tilde \eta(t,u)$ for $t \in [0,r)$. Hence, $\eta(t,u) \in B \cap S_u = B_\mu$ for all $t \in [0,r)$ and this completes the proof.
\end{proof}

The following corollary of Proposition \ref{propbmmanifold} will be crucial for our problem.

\begin{corollary} \label{corbm}
	Let $B_\mu \subset S_\mu$ be  closed in $S_\mu$. 
	Suppose that $V$ is a locally Lipschitz mapping with $V(u) \in T_uS_\mu$ for any $u \in S_\mu$ and define $\eta(t,\cdot): S_\mu \to S_\mu$ by
	\begin{align*}
		\left\{
		\begin{aligned}
			& \frac{\partial}{\partial t}\eta(t,u) = V(\eta(t,u)), \\
			& \eta(0,u) = u.
		\end{aligned}
		\right.
	\end{align*}
    Let $T_u > 0$ be the maximal time such that $\eta(\cdot,u)$ exists for any $u \in S_\mu$ and $B^* \subset S_\mu$ be an open subset such that $\eta(t,u) \in B^*$ for any $u \in B^*$ and all $0 < t < T_u$. Then, if
	\begin{align} \label{eqbm}
		\lim_{s \searrow 0}s^{-1}dist(u + sV(u),B_\mu) = 0, \quad \forall u \in B_\mu \cap B^*,
	\end{align}
	it holds that $\eta(t,u) \in B_\mu \cap B^*$ for any $u \in B_\mu \cap B^*$ and all $0 < t < T_u$.
\end{corollary}

\begin{proof}
	For any $u \in B_\mu \cap B^*$, since $B^*$ is open and $V$ is a locally Lipschitz mapping, we can choose a neighborhood $U$ of $u$ in $S_\mu$ such that $\overline{U} \subset B^*$ and $V$ is Lipschitz on $\overline{U}$. Also, it is clear that $B_\mu \cap \overline{U}$ is closed since $B_\mu$ is (locally) closed in $S_\mu$ and that \eqref{eqbm} implies
	\begin{align*}
		\lim_{s \searrow 0}s^{-1}dist(w + sV(w),B_\mu) = 0, \quad \forall w \in \overline{U} \cap B_\mu.
	\end{align*}
	Proposition \ref{propbmmanifold} then yields that $\eta(t,u) \in B_\mu$ for all $t \in [0,r)$ for some $0 < r \leq T_u$. By the assumption on $B^*$, we know $\eta(t,u) \in B^*$ for all $0 < t < T_u$. Let
	\begin{align*}
		r^* := \sup\bigl\{T: 0 < T < T_u, \eta(t,u) \in B_\mu \text{ in } t \in [0,T)\bigr\}.
	\end{align*}
    Obviously, $r* \geq r$. We claim that $r^* = T_u$. By contradiction, we suppose $r^* < T_u$. According to the definition of $r^*$, we know that $\eta(t,u) \in B_\mu$ for any $t \in [0,r^*)$. Using the closedness of $B_\mu$ we deduce that $\eta(r^*,u) \in  B_\mu$.  Since $\eta(r^*,u) \in  B^*$, that is, $\eta(r^*,u) \in  B_\mu \cap B^*$, using Proposition \ref{propbmmanifold} again we have $\eta(t,u) \in B_\mu$ in $t \in [0,r^* + \bar r)$ for some $0 < \bar r \leq T_u-r^*$. This contradicts the definition of $r^*$ and completes the proof.
\end{proof}

In the next lemma, we give a condition to check that hypothesis \eqref{eqd=0} (or  hypothesis  \eqref{eqbm}) holds.

\begin{lemma} \label{lemconvex}
	Let $B_\mu = \tilde B \cap S_\mu$ where $\tilde B \subset E$ is closed, convex in $E$ and satisfies the assumption:
	\begin{align} \label{eqconecondition}
		kw \in \tilde B, \quad \forall k \in (0,1), \ \forall w \in \tilde B.
	\end{align}
	For $u \in B_\mu$, if $V(u)$ has the form $G(u)-u$ with $G(u) \in \tilde B$, then 
	\begin{align} \label{eqd=0atapoint}
		\lim_{s \searrow 0}s^{-1}dist(u + sV(u),B_\mu) = 0.
	\end{align}
\end{lemma}

\begin{proof}
Given $s>0$, let
	 $u_s = u + sV(u)$ where $u \in B_\mu \subset S_\mu$. Since $V(u) \in T_uS_u$,
	\begin{align*}
		|u_s|^2 & = (u+sV(u),u+sV(u)) \\
		& = (u,u) + 2s(u,V(u)) + s^2(V(u),V(u)) \\
		& = \mu + s^2|V(u)|^2.
	\end{align*}
	For $s \in (0,1)$, we have $u_s = sG(u) + (1-s)u \in \tilde B$ by the convexity of $\tilde B$. Then, since $\sqrt \mu/|u_s| < 1$, using \eqref{eqconecondition} we know $\sqrt \mu u_s/|u_s| \in \tilde B \cap S_\mu = B_\mu$. Hence,
	\begin{align*}
		\lim_{s \searrow 0}s^{-1}dist(u + sV(u),B_\mu) & \leq \lim_{s \searrow 0}s^{-1}\bigl\|u_s -  \frac{\sqrt\mu u_s}{|u_s|}\bigr\| \\
		& = \lim_{s \searrow 0}\frac {s|V(u)|^2}{2\mu}\|u_s\| = 0,
	\end{align*}
	which completes the proof.
\end{proof}

%In our application with $E = H^1(\G)$, $H = L^2(\G)$, and $S_\mu = H^1_\mu(\G)$, if we choose $\tilde B = \p$ (or $\tilde B = -\p$), it is clear that $\tilde B$ satisfies the conditions in Lemma \ref{lemconvex}. However, it seems difficult to check the assumption $G(u) \in \tilde B$ for any \textcolor{red}{ $u \in \tilde{B}$.} In fact, in Lemma \ref{lemG(u)} proved below, we will show that $G(u) \in \pm \p$ for any $u \in \pm \p \cap \B_{\rho^*}^\mu$ rather than for any $u \in \pm \p \cap H^1_\mu(\G)$. \textcolor{red}{We proved in Lemma \ref{lemgloext}  that $\B^{M_2} \subset \B_{\rho^*}^\mu$} is an invariant subset, that is, $\eta(t,u) \in \B^{M_2}$ for any $u \in \B^{M_2}$ and all $t \in (0,\infty)$. For this reason, we \textcolor{red}{establish} a corollary of Proposition \ref{propbmmanifold} which is more \textcolor{red}{convenient for us.} 

%\begin{remark}
%	In Corollary \ref{corbm}, we can take $S_\mu$ as $B^*$ trivially. In this case, \textcolor{red}{ once} one has the global existence in time, the result of local trapping in $B_\mu$ of Proposition \ref{propbmmanifold} (namely \eqref{eqeta}) automatically extends from $t \in (0,r)$ to $t \in (0,T_u)$.
%\end{remark}

 Next we return to our specific problem where $E = H^1(\G)$, $H = L^2(\G)$, $S_\mu = H^1_\mu(\G)$. We shall use the above results with $B_\mu = (\pm\p)_\nu \cap H^1_\mu(\G)$, $\tilde B = (\pm\p)_\nu$ and $B^* = \B^{M_2}$.

Our next lemma is inspired by some elements found in the proof of \cite[Lemma 2.14]{TaTe}.

\begin{lemma} \label{lemG(u)}
 Let the gradient $\nabla E(\cdot,\G)$ be written in the form $\nabla E(u,\G) = u - G(u) \in T_uH^1_\mu(\G)$ for $u \in H^1_\mu(\G)$, where $T_uH^1_\mu(\G)$ is the tangent space of $H^1_\mu(\G)$ at $u$ defined as
	\begin{align*}
		T_uH^1_\mu(\G) := \bigl\{\phi \in H^1(\G): \int_\G u\phi dx = 0\bigr\}.
	\end{align*}
Then there exists $\nu^* > 0$ such that when $0 < \nu < \nu^*$, we have $G((\pm\p)_\nu \cap \B_{\rho^*}^\mu) \subset (\pm\p)_{\nu/2}$ for $0 < \mu \leq \tilde{\mu}$, where $\tilde{\mu}$ is defined in Proposition \ref{propinva}.
\end{lemma}

\begin{proof}
	By the definition of the gradient, $G(u)$ can be expressed as $(-\frac{d^2}{dx^2} + 1)^{-1}(|u|^{p-2}u + \la_uu)$, where $\la_u \in \R$ depends on $u$ and is determined by the condition $u - G(u) \in T_uH^1_\mu(\G)$, i.e.,
	\begin{align} \label{eqsignofla1}
		\int_{\G}(|u|^{p-2}u + \la_uu)(-\frac{d^2}{dx^2} + 1)^{-1}u \, dx = \int_{\G}|u|^2dx = \mu.
	\end{align}
    Let $\xi = (-\frac{d^2}{dx^2} + 1)^{-1}u$ for $u \in H^1_\mu(\G)$ and note that
    \begin{align} \label{eqsignofla2}
	\int_\G u\xi dx = \|\xi\|_{H^1(\G)}^2 > 0.
    \end{align}
	Recall that $\rho^* \leq b$ where $b = \left( \frac2{(p-2)K}\right)^{\frac4{p-6}}$. For any $u \in \B_{\rho^*}^\mu$, we have, using \eqref{ajout1}
	and that $-\xi'' + \xi =u$,
  %\textcolor{red}{using \eqref{equp},}
	\begin{align*}
		\int_\G |\xi|^pdx & \leq pK \left( \int_\G |\xi'|^2dx \right) ^{\frac{p-2}4} \left( \int_\G |\xi|^2dx \right)^{\frac{p+2}4} + p \ell^{\frac{2-p}2}\left( \int_\G |\xi|^2dx \right)^ {\frac p2} \nonumber\\
		& \leq pK \left( \int_\G |u'|^2dx \right) ^{\frac{p-2}4} \left( \int_\G |u|^2dx \right)^{\frac{p+2}4} + p \ell^{\frac{2-p}2}\left( \int_\G |u|^2dx \right)^ {\frac p2}.
	\end{align*}	
   Then, using H\"older inequality and \eqref{ajout1} again, we obtain
	\begin{align} \label{eqestiofla}
		\bigg|\int_\G|u|^{p-2}u\xi dx\bigg| & \leq \left( \int_\G|u|^pdx \right)^{\frac{p-1}p} \left( \int_\G|\xi|^pdx \right)^{\frac{1}p} \nonumber \\
		& \le pK\mu^{\frac{p+2}4}(\rho^*)^{\frac{p-2}4} + p\ell^{\frac{2-p}2}\mu^{\frac p2} \nonumber \\
		& \leq  pKb^{\frac{p-2}4}\mu^{\frac{p+2}4} +  p\ell^{\frac{2-p}2}\mu^{\frac p2}.
	\end{align} 
    When $0 < \mu \leq \tilde{\mu}$, \eqref{eqestiofla} shows that
    \begin{align} \label{eqsignofla3}
    	\int_\G|u|^{p-2}u\xi dx \leq \mu, \quad \forall u \in \B_{\rho^*}^\mu.
    \end{align}
	By \eqref{eqsignofla1}, \eqref{eqsignofla2} and \eqref{eqsignofla3}, we obtain that  $\la_u \ge 0$ for any $u \in \B_{\rho^*}^\mu$ if $0 < \mu \leq \tilde{\mu}$. Moreover, we claim that $\la_u \leq \bar \la$ for some $\bar \la < \infty$ independent of $u \in \B_{\rho^*}^\mu$. Using \eqref{eqsignofla1}, \eqref{eqsignofla2} and \eqref{eqsignofla3} again, it suffices to show the existence of  $\epsilon_0 > 0$ independent of $u \in \B_{\rho^*}^\mu$ such that $\|\xi\|_{H^1(\G)} \geq \epsilon_0$. By contradiction, we assume that there exists a sequence $\{u_n\} \subset \B_{\rho^*}^\mu$ such that $\xi_n \to 0$ strongly in $H^1(\G)$. Up to a subsequence, we can suppose that 
	\begin{align*}
		& u_n \rightharpoonup \bar u \quad \text{weakly in } H^1(\G), 
		& u_n \to \bar u \quad \text{strongly in } L^2(\G),
	\end{align*}
    for some $\bar u \in \overline{\B_{\rho^*}^\mu}$. Let $\bar \xi = (-\frac{d^2}{dx^2} + 1)^{-1}\bar u \neq 0$. On one hand, $\bar \xi \neq 0$ since $\bar u \neq 0$. On the other hand, $\{\xi_n\}$ strongly converges to $\bar \xi$ in $H^1(\G)$. Indeed,
    \begin{align} \label{eqconvergenceofxi}
    	\|\xi_n - \bar \xi\|_{H^1(\G)}^2 = \int_\G (u_n-\bar u)(\xi_n - \bar \xi)dx \leq \left( \int_\G(u_n-\bar u)^2dx\right) ^{\frac12}\left( \int_\G(\xi_n - \bar \xi)^2dx\right) ^{\frac12} \to 0.
    \end{align}
     Next, let $u \in (-\p)_\nu \cap \B_{\rho^*}^\mu$, $z = G(u)$, $z^+ = \max\{z,0\}$. Then
	\begin{align*}
		& \quad \ dist(z,-\mathcal{P}) \|z^+\|_{H^1(\G)} \leq \|z^+\|_{H^1(\G)}^2 = \langle z,z^+\rangle_{H^1(\G)} = \int_\G  z^+ ( -\frac{d^2}{dx^2} + 1) zdx\\
		& =  \int_\G (|u|^{p-2}u + \la_uu)z^+dx \leq \int_\G (|u^+|^{p-2}u^+z^+ + \la_uu^+z^+)dx \\ 
		& \leq \left( \int_\G|u^+|^pdx\right)^{\frac{p-1}p} \left( \int_\G|z^+|^pdx\right)^{\frac1p} + \bar \la |\{x \in \G: z(x) > 0\}|^{\frac{p-2}{p}} \left( \int_\G|u^+|^pdx\right) ^{\frac1p} \left( \int_\G|z^+|^pdx\right)^{\frac1p} \\
		& \leq C_1\left( \inf_{w \in -\p}\|u-w\|_{L^p(\G)}^{p-1} + \bar \la |\{x \in \G: z(x) > 0\}|^{\frac{p-2}{p}}\inf_{w \in -\p}\|u-w\|_{L^p(\G)}\right) \|z^+\|_{H^1(\G)} \\
		& \le C_2 \left( \inf_{w \in -\p}\|u-w\|_{H^1(\G)}^{p-1} +  \bar \la |\{x \in \G: z(x) > 0\}|^{\frac{p-2}{p}}\inf_{w \in -\p}\|u-w\|_{H^1(\G)}\right) \|z^+\|_{H^1(\G)},
	\end{align*}
    where $C_1,C_2$ are positive constants deduced from the Sobolev embedding $H^1(\G) \hookrightarrow L^p(\G)$, and we obtain that
    \begin{align} \label{eqdistofvandp}
    	dist(v,-\mathcal{P}) \leq C_2\left(\nu^{p-2} + \bar \la |\{x \in \G: z(x) > 0\}|^{\frac{p-2}{p}}\right) \nu.
    \end{align}
    In order to conclude that $z = G(u) \in (-\p)_{\nu/2}$ with sufficiently small $\nu > 0$, it is sufficient to prove that $\sup\{|\{x \in \G: G(u)(x) > 0\}|: u \in (-\p)_\nu \cap \B_{\rho^*}^\mu\}$ is sufficiently small when $\nu > 0$ is small enough. Arguing by contradiction, we suppose the existence of sequences $\nu_n \to 0$ and  $\{u_n\} \subset \B_{\rho^*}^\mu$ such that $u_n \in (-\p)_{\nu_n}$ and 
    \begin{align} \label{eqlim>0}
    	\lim_{n \to \infty}|\{x \in \G: G(u_n)(x) > 0\}| > 0.
    \end{align}
    Up to a subsequence, we have 
    \begin{align*}
    	& u_n \rightharpoonup \bar u \quad \text{weakly in } H^1(\G), 
    	& u_n \to \bar u \quad \text{strongly in } L^q(\G) \text{ for any } q \geq 1,
    \end{align*}
    for some $\bar u \in \overline{\B_{\rho^*}^\mu}$. Setting $\xi_n = (-\frac{d^2}{dx^2} + 1)^{-1}u_n$ and $\bar\xi = (-\frac{d^2}{dx^2} + 1)^{-1}\bar u$,  we have, as in  \eqref{eqconvergenceofxi}, that $\xi_n \to \bar \xi$ strongly in $H^1(\G)$. Then, in view of \eqref{eqsignofla1} we get $\la_{u_n} \to \la_{\bar u}$. Now, since $|u_n|^{p-1}u_n \to |u|^{p-1}u$ strongly in $L^2(\G)$, by a similar argument to \eqref{eqconvergenceofxi} we deduce that $G(u_n) \to G(\bar u)$ strongly in $H^1(\G)$. Since $u_n \in (-\p)_{\nu_n}$ with $\nu_n \to 0$, we conclude that $\bar u \leq 0$, and so $(-\frac{d^2}{dx^2} + 1)G(\bar u) = |\bar u|^{p-2}\bar u + \la_{\bar u}\bar u \leq 0$. Moreover, by \eqref{eqdistofvandp} we have $G(\bar u) \in -\p$, that is, $ G(\bar u) \leq 0$. Then, using the strong maximum principle we get $G(\bar u) < 0$ in $\G$, which, together with the compactness of $\G$, shows that
    \begin{align*}
    	\lim_{n \to \infty}|\{x \in \G: G(u_n)(x) > 0\}| = 0,
    \end{align*}
    contradicting \eqref{eqlim>0}. Thus $G(u) \in (-\p)_{\nu/2}$ if $\nu > 0$ is sufficiently small. A similar argument works for the $+$ sign.
\end{proof}

Now we can prove the main result in this section.

\begin{proof}[Proof of Proposition \ref{propinva}]
Given $\nu \in (0,\nu^*)$, we use Corollary \ref{corbm} with $B_\mu = (\pm\p)_\nu \cap H^1_\mu(\G)$ and $B^* = \B^{M_2}$. The fact that \eqref{eqbm} is satisfied is a consequence of Lemma \ref{lemconvex} used with $\tilde{B} = (\pm\p)_\nu$. Indeed, we know from Lemma \ref{lemG(u)} that $G(u) \in (\pm\p)_{\nu/2}$ if $u \in (\pm\p)_\nu \cap \B^{M_2}$ and thus \eqref{eqd=0atapoint}, or equivalently \eqref{eqbm}, holds for any $u \in (\pm\p)_\nu \cap \B^{M_2} = B_{\mu} \cap \B^{M_2}.$  Applying Corollary \ref{corbm} we obtain that  $\eta(t,u) \in D^*(\nu) \cap \B^{M_2}$ for any $u \in D^*(\nu) \cap \B^{M_2}$ and all $0 < t <T_u$. However, Lemma \ref{lemgloext} states that if $u \in \B^{M_2}$, then $\eta(t,u)$ exists for all $t > 0$ with 
$\eta(t,u) \in \B^{M_2}$ for all $t > 0$. Thus we finally get that, for all $u \in D^*(\nu) \cap \B^{M_2}$,  $\eta(t,u) \in D^*(\nu) \cap \B^{M_2}$ for all $t>0$, completing the proof.
\end{proof}

	\begin{remark}
		The value of $\nu^*$ in Proposition \ref{propinva} is independent of the choice of $\delta_1$ used in defining $\eta$. Thus, in the proof of Theorem \ref{thmsc}, in the next section, for $0 < \nu < \nu^*$, it will be not restrictive to take $\delta_1 < \nu$.
	\end{remark}

\section{A pair of sign-changing critical points} \label{secproof}

We will complete the proof of Theorem \ref{thmsc} in this section. 

\begin{proof} [Proof of Theorem \ref{thmsc}]
	
Let $\mu_2 = \min\{\hat{\mu}, \bar{\mu}, \mu^*, \tilde{\mu}\}$. With $\nu > 0$ sufficiently small, we first prove that $\K[c_1-\bar \epsilon,c_1+\bar \epsilon] \cap \B^{M_2} \cap S^*(\nu) \neq \emptyset$ for $0 < \mu < \mu_2$ and for any $\bar \epsilon > 0$, where
$$ c_1 = \inf_{A \in \cl_1}\sup_{u \in A\cap S^*(\nu)}E(u,\G) \quad %\mbox{and} \quad \bar{c}_1 = \inf_{u\in\s_1^\perp \cap \B^{M_2}}E(u,\G)
$$
was defined in Section \ref{seclink}. 
%\begin{align*}
%	& c_1 = \inf_{A \in \cl_1}\sup_{u \in A\cap S^*}E(u,\G), \\
%	& \bar{c}_1 = \inf_{u\in\s_1^\perp \cap \B^{M_2}}E(u,\G),
%\end{align*}
%were well defined in Section \ref{seclink}. 

Before proceeding, we give a lemma, which is useful to choose $\delta_1$ used to define the descending flow.

\begin{lemma} \label{lembelongtog}
	For any $u \in \B^\beta$ with $\beta < M_2$, there exists $\tilde \delta > 0$ independent of $u$ such that
	\begin{align}
		B_{\tilde \delta}(u) := \bigl\{w \in H^1_\mu(\G): \|u - w\|_{H^1(\G)} < \tilde \delta\bigr\} \subset \B^{M_2}.
	\end{align}
\end{lemma}
\begin{proof}
	Note that the functions in $\B^\beta$ are uniformly bounded in $H^1(\G)$. Hence, for $u \in \B^\beta$ with $\beta < M_2$, there exists $\tilde \delta > 0$ independent of $u$ such that $\sup_{w \in B_{\tilde \delta}(u)}E(w,\G) < M_2$. Since $\inf_{u \in \U_{\rho^*}^\mu}E(u,\G) \geq M_2$ where $\U_{\rho^*}^\mu$ is the boundary of $\B_{\rho^*}^\mu$ and $\B_{\rho^*}^\mu$ is open, we deduce that $B_{\tilde \delta}(u) \subset \B_{\rho^*}^\mu$. Therefore, $B_{\tilde \delta}(u) \subset \B^{M_2}$, which completes the proof.
\end{proof}

To prove Theorem \ref{thmsc}, we argue by contradiction and suppose there exists some $\bar{\epsilon} > 0$ such that $\K[c_1-\bar{\epsilon},c_1+\bar{\epsilon}] \cap \B^{M_2} \subset D^*(\nu)$ (this also holds if $\K[c_1-\bar{\epsilon},c_1+\bar{\epsilon}] \cap \B^{M_2} = \emptyset$). Then, assuming $\nu < \nu^*$ and using Lemma \ref{lemG(u)}, we obtain that $\K[c_1-\bar{\epsilon},c_1+\bar{\epsilon}] \cap \B^{M_2} \subset \p \cup (-\p)$. By the definition of $c_1$, there exist $A \in \mathcal{L}_1$ and $0 < \epsilon_1 \leq \bar{\epsilon}/3$ such that $\sup_{A \cap S^*(\nu)}E(\cdot,\G) < c_1 + \epsilon_1 < M_2$. By Lemma \ref{lemvaluesepe}, we know that $\underline{c}_1 < \bar{c}_1$. Hence, by decreasing $\epsilon_1$, we can assume that $3\epsilon_1 \leq \bar{c}_1 - \underline{c}_1 \leq c_1 - \underline{c}_1$. Note that $A \subset (\B^{c_1 + \epsilon_1} \cap S^*(\nu)) \cup (\B^{M_2} \cap D^*(\nu))$. Now we take $\delta_1 > 0$ such that $\delta_1 < \nu$ and $\delta_1 \leq \tilde \delta$, where $\tilde \delta$ is given in Lemma \ref{lembelongtog} with $\beta := c_1 + \epsilon_1 < M_2$. For any $u \in \B^{c_1 + \epsilon_1} \cap S^*(\nu)$, we claim that $dist(u,\K[c_1-\bar{\epsilon},c_1+\bar{\epsilon}]) \geq \delta_1$. Otherwise, there exists $w \in \K[c_1-\bar{\epsilon},c_1+\bar{\epsilon}]$ such that $\|u - w\|_{H^1(\G)} < \delta_1$. Since $\delta_1 < \tilde \delta$, using Lemma \ref{lembelongtog} we deduce that $w \in \B^{M_2}$, and so $w \in \p \cup (-\p)$. Observing that $u \in S^*(\nu)$, we have reached a contradiction since $\delta_1 < \nu$. 

Let $\eta(t,\cdot)$ be the decreasing flow defined in Section \ref{secinv}.
We claim that
\begin{equation}\label{ajout12}
\eta(2\epsilon_1,A) \subset D^*(\nu) \cup \B^{c_1 - \epsilon_1}. 
\end{equation}
In fact, for $u \in A$, if $\eta(t,u) \in D^*(\nu)$ and so $\eta(t,u) \in D^*(\nu) \cap \B^{M_2}$ for some $t \in [0,2\epsilon_1]$, Proposition \ref{propinva} yields that $\eta(2\epsilon_1,u) \in D^*(\nu)$. If $\eta(t,u) \in \B^{c_1 - \epsilon_1}$ for some $t \in [0,2\epsilon_1]$, since it is readily seen that $E(\eta(t,u),\G)$ is nonincreasing in $t \geq 0$, we obtain that $\eta(2\epsilon_1,u) \in \B^{c_1 - \epsilon_1}$.  If we are not in one of these two cases, necessarily 
\begin{equation}\label{detail}
\eta(t,u) \in S^*(\nu) \cap (\B^{c_1+\epsilon_1} \backslash \B^{c_1 - \epsilon_1}), \quad \mbox{for all } t \in [0,2\epsilon_1].
\end{equation}
Let us show, by contradiction, that \eqref{detail} is not true. This will prove that \eqref{ajout12} holds. Assuming \eqref{detail}, we observe that for $t \in [0,2\epsilon_1]$, $y(\eta(t,u)) = 1$ since we have proved that $dist(\eta(t,u),\K[c_1-\bar{\epsilon},c_1+\bar{\epsilon}]) \geq \delta_1$ when $\eta(t,u) \in \B^{c_1+\epsilon_1} \cap S^*(\nu)$. Also $h(\eta(t,u)) = 1$ since $\eta(t,u) \in U_2$. Then we have
\begin{align*}
	E(\eta(2\epsilon_1,u),\G) & = E(u,\G) + \int_0^{2\epsilon_1}dE(\eta(t,u),\G) \\
	& = E(u,\G) + \int_0^{2\epsilon_1} \bigl\langle \nabla E(\eta(t,u),\G),\frac{\partial}{\partial t}\eta(t,u) \bigr\rangle_{H^1(\G)} \\
	& = E(u,\G) - 2\epsilon_1 \\
	& < c_1 + \epsilon_1 - 2\epsilon_1 \\
	& = c_1 - \epsilon_1,
\end{align*}
which implies that $\eta(2\epsilon_1,u) \in \B^{c_1-\epsilon_1}$ and this contradicts  \eqref{detail}. Thus \eqref{ajout12} holds.

We will now show that \eqref{ajout12} is impossible. This will prove that the hypothesis that there exists a certain $\bar{\epsilon} > 0$ such that $\K[c_1-\bar{\epsilon},c_1+\bar{\epsilon}] \cap \B^{M_2} \subset D^*(\nu)$ was wrong. Since $\epsilon_1 > 0$ is chosen as $3\epsilon_1 \leq \bar{c}_1 - \underline{c}_1 \leq c_1 - \underline{c}_1$, it holds that $\mathbb{S}_1 \subset U_1$ where $U_1$ is given in the definition of $\eta(t,\cdot)$, and so
$\eta(t,\cdot)|_{\mathbb{S}_1} = \textbf{id}$ for all $t \geq 0$. Hence, $\eta(2\epsilon_1,\cdot)$ defines a continuous function from $A$ to $H^1_\mu(\G)$ with
$\eta(2\epsilon_1,\cdot)|_{\mathbb{S}_1} = \textbf{id}$. We will show that $\eta(2\epsilon_1,A) \in \mathcal{L}_1$, for which, it suffices to prove that $\eta(2\epsilon_1,A)$ is linked to $\mathbb{S}_1^{\perp}$. Let $f \in C(\eta(2\epsilon_1,A),H_\mu^1(\G))$ satisfying $f|_{\mathbb{S}_1} = \textbf{id}$. Then $f \circ \eta(2\epsilon_1,\cdot)$ is a continuous map from $A$ to $H_\mu^1(\G)$ satisfying $f \circ \eta(2\epsilon_1,\cdot)|_{\mathbb{S}_1} = \textbf{id}$, and there exists $u \in f \circ \eta(2\epsilon_1,A) \cap \mathbb{S}_1^{\perp}$, that is, $f(\eta(2\epsilon_1,A)) \cap \mathbb{S}_1^{\perp} \neq \emptyset$. At this point, using $\eta(2\epsilon_1,A) \in \mathcal{L}_1$ and the definition of $c_1$ it follows that \eqref{ajout12} is indeed impossible.

%Since $A \in \mathcal{L}_1$, we have $\eta(T_A,A) \cap \mathbb{S}_1^{\perp} \not= \emptyset$, i.e., there is $u^* \in \eta(T_A,A) \cap \mathbb{S}_1^{\perp}$. \textcolor{red}{Since $u^* \in \mathbb{S}_1^{\perp} \cap B_{\rho^*}^{\mu}$, in view of Lemma \ref{lemfarawayfrom}, $u^* \notin (D^*)_{\delta_1/2}$. Also, still using that $u^* \in \mathbb{S}_1^{\perp}$, and that, by assumption, $\mathcal{K}[\bar{c}_1-\bar{\epsilon},c_1+\bar{\epsilon}] \cap \B^{M_2} \subset D^*$, Lemma \ref{lemfarawayfrom} and Lemma \ref{lembelongtog} yield that $u^* \notin (\mathcal{K}[\bar{c}_1-\bar{\epsilon},c_1+\bar{\epsilon}] )_{2\delta_1/3}\cap \B^{M_2}$  ($\delta_1 > 0$ is chosen smaller than $\delta = dist(\mathbb{S}_{1}^\perp \cap \B_{\rho^*}^\mu, D^*) > 0$ and $B_{\delta_1}(u^*) \subset \B^{M_2}$). Finally, $E(u^*,\G) \geq \inf_{\mathbb{S}_1^{\perp}\cap \B^{M_2}}E = \bar{c}_1$, implies that $u^* \notin B^{\bar{c}_1-\epsilon_1}$. 
%We have thus reached  a contradiction to \eqref{ajout12}.} 
Therefore, for $0 < \mu < \mu_2$ and any $\bar \epsilon > 0$, $E(u,\G)$ has a constrained critical point $u_{2,\bar \epsilon}$ in $\B^{M_2} \cap S^*(\nu)$ with $E(u_{2,\bar \epsilon},\G) \in [c_1 - \bar{\epsilon},c_1 + \bar{\epsilon}]$. By Corollary \ref{corpscon}, we can get a critical point $u_2$ at level $c_1$. Moreover, $u_2$ is contained in $\overline{S^*(\nu)}$ and thus is sign-changing. Obviously, $u_2$ is non-constant and $-u_2 \not\equiv u_2$ is also a sign-changing critical point of $E(u,\G)$ constrained to $H^1_\mu(\G)$.

\medskip
For $0 < \mu_1$ defined in \eqref{defmu1}, we know from \cite[Theorem 1.1]{CJS} that $E(\cdot,\G)$ has, for $0 < \mu < \mu_1$, a positive critical point $u_1$ in $H^1_\mu(\G)$ at a mountain pass level. Obviously, $-u_1$ is a negative critical point at the same mountain pass level. Thus, if  $\mu_2 \leq \mu_1$, we obtain the desired result in Theorem \ref{thmsc}. If $\mu_2 > \mu_1$, we just redefine $\mu_2 := \min\{\mu_1,\hat{\mu}, \bar{\mu}, \mu^*, \tilde{\mu}\}$. The proof of Theorem \ref{thmsc} is now completed.
\end{proof}

\begin{remark}
It seems difficult to compare $\mu_1 >0$ for which a positive solution was obtained in \cite{CJS} for any $0 < \mu < \mu_1$ and the
 $\mu_2 >0$ of Theorem \ref{thmsc} since we do not know the expressions (or suitable estimates) of the best constant $K$ of the Gagliardo-Nirenberg inequality \eqref{eqg-n1} and the second eigenvalue $\la_2(\G)$ for general compact metric graphs. 
\end{remark}

\section{More and more sign-changing critical points by decreasing the mass if necessary and a bifurcation phenomena} \label{secproof2}

The goal of this section is to give the proof of Theorem \ref{thmsc2} and of Theorem \ref{thmsc3}. As the proof of Theorem \ref{thmsc2} is quite similar to the one of Theorem \ref{thmsc} we will not give the full details.

\begin{proof}[Proof of Theorem \ref{thmsc2}]
	We divide the proof into four steps.
	
	\vskip0.1in
	\noindent {\bf Step 1:} The min-max structures.
	
	Let $0=\la_1(\G) < \la_2(\G) \leq \la_3(\G) \leq \cdots $ be the eigenvalues of $-\frac{d^2}{dx^2}$ on the compact metric graph $\G$ with Kirchhoff condition at the vertices, and let $\phi_1, \phi_2, \phi_3, \cdots$ be the corresponding eigenfunctions with $\int_\G|\phi_i|^2dx = 1$ for all $i \in \N$. For any $k \in \N$ with $k \geq 2$ and $\la_{k-1}(\G) < \la_k(\G)$, we set
	\begin{align*}
		\s_{k-1} = span\{\phi_1,\cdots,\phi_{k-1}\} \cap H^1_\mu(\G), \quad \s_{k-1}^\perp = (span\{\phi_1,\cdots,\phi_{k-1}\})^\perp \cap H^1_\mu(\G).
	\end{align*}
    Note that
    \begin{align*}
    	\s_{k-1} = \left\{u = \sum_{i=1}^{k-1}t_i\sqrt{\mu}\phi_i: \sum_{i=1}^{k-1}t_i^2 = 1\right\}.
    \end{align*}
    We also introduce
    \begin{align*}
    	Q_k := \left\{u = \sum_{i=1}^{k}t_i\sqrt{\mu}\phi_i: \sum_{i=1}^{k}t_i^2 = 1, t_k \geq 0\right\},
    \end{align*}
    with
    \begin{align*}
    	\partial Q_k = \left\{u = \sum_{i=1}^{k-1}t_i\sqrt{\mu}\phi_i: \sum_{i=1}^{k-1}t_i^2 = 1\right\} = \mathbb{S}_{k-1}.
    \end{align*}
    Similarly to \eqref{eqdefl1}, we define
    \begin{align*} 
		%\label{eqdeflk-1}
    	\cl_{k-1} = \cl_{k-1}(\mu) := \bigl\{A \subset \B_{\rho^*}^\mu: \sup_A E(\cdot,\G) < M_2, \ \s_{k-1} \subset A, \ A \text{ is compact and is linked to } \s_{k-1}^\perp\bigr\}.
    \end{align*}
	Let $\hat{\mu}_k = \hat{\mu}_k(\G,p) > 0$ be the positive solution of
	$$
	\la_k(\G)\mu + 2\ell^{\frac{2-p}2}\mu^{\frac p2} = \frac{p-6}{p-2}\left( \frac2{(p-2)K}\right)^{\frac4{p-6}},
	$$
	and let $\bar{\mu}_k = \bar{\mu}_k(\G,p) > 0$ be the positive solution of
	$$
	\la_k(\G)\mu^{\frac{2(p-2)}{p-6}} + 2\ell^{\frac{2-p}2}\mu^{\frac{(p-2)^2}{2(p-6)}} = \frac{p-6}{p-2}\left( \frac2{(p-2)K}\right)^{\frac4{p-6}}.
	$$
	Reasoning as in Lemma \ref{lemlnotemp}, we observe that $Q_k \subset \cl_{k-1}$ and thus $\cl_{k-1}$ is not empty for $0 < \mu < \min\{\hat{\mu}_k,\bar{\mu}_k\}$. Recalling that
		\begin{align*}
			& \pm \p = \bigl\{u \in H^1(\G): \pm u \geq 0\bigl\},\\
			%& \ \K = \bigl\{u \in H^1_\mu(\G): (E|_{H^1_\mu(\G)})'(u,\G) = 0\bigl\},\\
			%& \ \K[\alpha,\beta] = \bigl\{u \in H^1_\mu(\G): (E|_{H^1_\mu(\G)})'(u,\G) = 0, \alpha \leq E(u) \leq \beta\bigr\}. \\
			& (\pm \p)_\nu = \bigl\{u \in H^1(\G): \inf_{w \in \pm \p}\|u -w\|_{H^1(\G)} \le \nu\bigr\}, \quad \text{for some } \nu > 0,  \\
			& \ D^*(\nu) = \p_\nu \cup (-\p)_\nu, \quad S^*(\nu) = H^1(\G) \backslash D^*(\nu),
		\end{align*}	
    and that
    \begin{align*}
    	\B^{M_2} := \bigl\{u \in \B_{\rho^*}^\mu: E(u,\G) < M_2\bigr\},
    \end{align*}
	we define
	\begin{align*}
		& c_{k-1} := \inf_{A \in \cl_{k-1}}\sup_{u \in A\cap S^*(\nu)}E(u,\G), \\
		& \bar{c}_{k-1} := \inf_{u\in\s_{k-1}^\perp \cap \B^{M_2}}E(u,\G), \\
		& \underline{c}_{k-1} := \sup_{u \in \s_{k-1}}E(u,\G).
	\end{align*}
	Reasoning as in Lemma \ref{lemfarawayfrom}, we have that $\delta := \min\{dist(\mathbb{S}_{k-1}^\perp \cap \B_{\rho^*}^\mu, \p), dist(\mathbb{S}_{k-1}^\perp \cap \B_{\rho^*}^\mu, -\p)\} > 0$. As in Lemma \ref{lemwelldef}, for $0 < \nu < \delta$, we can show that $c_{k-1}$ and $\bar{c}_{k-1}$ are well defined and that $ - \infty <\bar{c}_{k-1} \leq c_{k-1} < M_2$ for $0 < \mu < \min\{\hat{\mu}_k,\bar{\mu}_k\}$. Also, as in Lemma \ref{lemvaluesepe}, using that $\lambda_{k-1}(\G) < \lambda_k(\G)$, it follows that  $\bar{c}_{k-1} > \underline{c}_{k-1}$ for $0 < \mu < \min\{\hat{\mu}_k,\bar{\mu}_k, \mu^*_k\}$ where $\mu^*_k = \mu^*_k(\G,p) > 0$ is the positive solution of
	$$
	2\left( \ell^{\frac{2-p}2} + K\la_k(\G)^{\frac{p-2}4}\right) \mu^{\frac {p-2}2} = \la_k(\G) - \la_{k-1}(\G).
	$$

	\vskip0.1in
	\noindent {\bf Step 2:} Invariance under the negative gradient flow.
	
	Fix $0 < \delta_{k-1} < \nu$ and consider the locally Lipschitz continuous map $y_{k-1}(u): H^1_\mu(\G) \to [0,1]$ defined by
	\begin{align*}
		y_{k-1}(u) :=
		\left\{
		\begin{aligned}
			& 1 \quad \text{for all } u \in H^1_\mu(\G) \backslash(\mathcal{K}[c_{k-1}-\bar{\epsilon},c_{k-1}+\bar{\epsilon}])_{\delta_{k-1}/2}, \\
			& 0 \quad \text{for all } u \in (\mathcal{K}[c_{k-1}-\bar{\epsilon},c_{k-1}+\bar{\epsilon}])_{\delta_{k-1}/3},
		\end{aligned}
		\right.
	\end{align*}
    where $\bar \epsilon >0$ will be chosen in Step 3 later. Moreover, let
		$$h_{k-1}(u):= \frac{dist(u,U_{1,k-1})}{dist(u,U_{2,k-1})+dist(u,U_{1,k-1})},$$
		where
    	\begin{align*}
    	U_{1,k-1} := \bigl\{u \in H^1_\mu(\G): |E(u, \G) - c_{k-1}| \geq 3 \epsilon_{k-1}\bigr\}, \quad U_{2,k-1} := 
			\bigl\{u \in H^1_\mu(\G): |E(u,\G) - c_{k-1}| \leq 2 \epsilon_{k-1}\bigr\}, 
    \end{align*}
    and $0< 3\epsilon_{k-1} \leq \bar{\epsilon}$ will be chosen in Step 3 below. Without loss of generality, we further assume that $\delta_{k-1} \leq \tilde{\delta}$ where $\tilde \delta$ is given in Lemma \ref{lembelongtog} with $\beta := c_{k-1} + \epsilon_{k-1} < M_2$. Let $\eta_k(t,\cdot)$ be the decreasing flow on $H^1_\mu(\G)$ defined by
	\begin{align*}
		\left\{
		\begin{aligned}
			& \frac{\partial}{\partial t}\eta_k(t,u) = -h_{k-1}(\eta_k(t,u))y_{k-1}(\eta_k(t,u))\|\nabla E(\eta_k(t,u),\G)\|_{H^1(\G)}^{-2}\nabla E(\eta_k(t,u),\G), \\
			& \eta_k(0,u) = u.
		\end{aligned}
		\right.
	\end{align*}
	As in Lemma \ref{lemgloext}, one can show that for any  $u \in \B^{M_2}$, $\eta_k(t,u)$ exists for $t \in [0,\infty)$, and, $\eta_k(t,u) \in \B^{M_2}$ for all $t > 0$.
	
	Recall that $\tilde{\mu} = \tilde{\mu}(\G,p)$ is the positive solution of 
	\begin{align*}
		pK\left( \frac2{(p-2)K}\right)^{\frac{p-2}{p-6}}\mu^{\frac{p-2}4} + p\ell^{\frac{2-p}2}\mu^{\frac{p-2}2} = 1.
	\end{align*}
   Reasoning as in Proposition \ref{propinva}, one can show that, for sufficiently small $\nu > 0$, if $0 < \mu \leq \tilde{\mu}$, then $\eta_k(t,u) \in D^*(\nu) \cap \B^{M_2}$ for any $u \in D^*(\nu) \cap \B^{M_2}$ and all $t > 0$.
	
	\vskip0.1in
	\noindent {\bf Step 3:} Existence of a sign-changing critical point at level $c_{k-1}$.
	
	Setting $\check{\mu}_k = \min\{\hat{\mu}_k, \bar{\mu}_k, \mu^*_k, \tilde{\mu}\}$, and taking $\nu > 0$ sufficiently small, we aim to prove that $\K[c_{k-1}-\bar{\epsilon},c_{k-1}+\bar{\epsilon}] \cap \B^{M_2} \cap S^*(\nu) \neq \emptyset$ for $0 < \mu < \check{\mu}_k$ and any $\bar{\epsilon} > 0$. By contradiction, we suppose there exists some $\bar{\epsilon} > 0$ such that $\K[c_{k-1}-\bar{\epsilon},c_{k-1}+\bar{\epsilon}] \cap \B^{M_2} \subset D^*(\nu)$ (this also holds if $\K[c_{k-1}-\bar{\epsilon},c_{k-1}+\bar{\epsilon}] \cap \B^{M_2} = \emptyset$) for $0 < \mu < \check{\mu}_k$. By the definition of $c_{k-1}$, there exist $A \in \mathcal{L}_{k-1}$ and $0 < \epsilon_{k-1} \leq \bar{\epsilon}/3$ such that $\sup_{A \cap S^*(\nu)}E < c_{k-1} + \epsilon_{k-1} < M_2$. Since $\underline{c}_{k-1} < \bar{c}_{k-1}$, by decreasing $\epsilon_{k-1}$, we can assume that $3\epsilon_{k-1} \leq \bar{c}_{k-1} - \underline{c}_{k-1} \leq c_{k-1} - \underline{c}_{k-1}$. By arguments similar to the ones in the proof of Theorem \ref{thmsc}, we can obtain that
	\begin{align*} 
	%\label{eqsubk-1}
		\eta_k(2\epsilon_{k-1},A) \subset D^*(\nu) \cup \B^{c_{k-1}-\epsilon_{k-1}}.
	\end{align*}
   Then, as in Section \ref{secproof}, we can prove that $\eta_k(2\epsilon_{k-1},A) \in \cl_{k-1}$ and find a contradiction with the definition of $c_{k-1}$. %existence of a $u^* \in \eta(T_A,A) \cap \mathbb{S}_{k-1}^{\perp}$ which should also satisfy
%	\textcolor{red}{
%    \begin{align*}
 %   	u^* \notin  (D^*)_{\delta_{k-1}/2} \cup \B^{\bar{c}_{k-1}-\epsilon_{k-1}} \cup \left(  (\mathcal{K}[\bar{c}_{k-1}-\bar{\epsilon},c_{k-1}+\bar{\epsilon}])_{2\delta_{k-1}/3} \cap \B^{M_2}\right).
%    \end{align*}
%		}
    It completes this step and proves that for $0 < \mu < \check{\mu}_k$ and for any $\bar{\epsilon} > 0$, there is a constrained critical point $u_{k,\bar \epsilon}$ in  $\B^{M_2} \cap S^*(\nu)$ of $E(\cdot,\G)$ with $E(u_{k,\bar \epsilon},\G) \in [c_{k-1}-\bar{\epsilon},c_{k-1}+\bar{\epsilon}]$. By Corollary \ref{corpscon}, we can get a critical point $u_k$ at level $c_{k-1}$. Moreover, $u_k$ is contained in $\overline{S^*(\nu)}$ and thus is sign-changing.
	
	\vskip0.1in
	\noindent {\bf Step 4:} Distinguishing different critical points.
	
	We rewrite the eigenvalues of $-\frac{d^2}{dx^2}$ on the compact metric graph $\G$ with Kirchhoff condition at the vertices as
	\begin{align*}
		0=\la_1(\G) < \la_2(\G) \leq \la_3(\G) \leq \cdots \leq \la_{k_3-1}(\G) < \la_{k_3}(\G) \leq \cdots \leq \la_{k_n-1}(\G) < \la_{k_n}(\G) \leq \cdots
	\end{align*}
	For any $j \in \N$ with $j \geq 3$, we can take $\{k_2, k_3,\cdots,k_j\}$ such that $2 = k_2 < k_3 < \cdots < k_j$ and $\la_{k_i}(\G) > \la_{k_i-1}(\G)$ for any $i \in \{2,3,\cdots,j\}$. Let 
	$$
	\mu_j = \min\bigl\{\mu_1, \check{\mu}_{k_2},\cdots,\check{\mu}_{k_j}\bigr\}.
	$$
	By Steps 1, 2, 3, for $0 < \mu < \mu_j$, there exist $j-1$ sign-changing critical points $u_{k_2}, u_{k_3},\cdots, u_{k_j}$ of $E(\cdot,\G)$ constrained on $H^1_\mu(\G)$ such that $E(u_{k_i},\G) = c_{k_i-1}$ for any $i \in \{2,3,\cdots,j\}$.  We need to show that these critical points are distinct. On the one hand, testing the definition of $c_{k_i-1}$ with the choice $A = Q_{k_i} \subset \s_{k_i}$, we get
	$$
	c_{k_i-1} \leq \sup_{u \in \s_{k_i}}E(u,\G) \leq \sup_{u \in \s_{k_{i+1}-1}}E(u,\G) = \underline{c}_{k_{i+1}-1}.
	$$
	On the other hand, from Step 1, we know that,
	$$
	c_{k_i-1} \geq \bar{c}_{k_i-1} > \underline{c}_{k_i-1}, \quad i \in \{2,3,\cdots,j\}.
	$$
	Thus we can conclude that
		$$
		c_1 = c_{k_2-1} \leq \underline{c}_{k_3-1} < c_{k_3-1} \leq \underline{c}_{k_4-1} < \cdots < c_{k_{j-1}-1} \leq \underline{c}_{k_j-1} < c_{k_j-1},
		$$	
	which yields that $u_{k_i} \not\equiv u_{k_n}$ for any $i \neq n$. Obviously, $-u_{k_2}, -u_{k_3},\cdots, -u_{k_j}$ are also sign-changing critical points of $E(\cdot,\G)$ constrained on $H^1_\mu(\G)$. Then, recalling \cite[Theorem 1.1]{CJS} completes the proof of Theorem \ref{thmsc2}.
\end{proof}

Finally, we can give

\begin{proof}[Proof of Theorem \ref{thmsc3}]
We denote by $B_0$ the set of bifurcation points of \eqref{eqequwithkc}. It is classical, see for example, \cite{CrRa}, that $B_0 \subset \{ \lambda_i(\G): i \in \N\}$ and that $\lambda_1(\G) \in B_0$ since $\lambda_1(\G)$ is simple. To prove the theorem it suffices thus to show that $\lambda_i(\G) \in B_0$ for any $i \geq 2$. Recalling our notation $0 =\lambda_1(\G) <\lambda_2(\G)= \cdots \la_{k_3 -1}(\G) < \la_{k_3}(\G) \leq \cdots $, we only need to show that $\la_{k_j}(\G)\in B_0$ for any $j \geq 2$. From now on, we fix an arbitrary $j \geq 2$. By the proof of Theorem \ref{thmsc2}, for any $\mu \in (0, \mu_j)$, $(u_{\mu ,j}, \la_{\mu, j}) \in H^1(\G) \times \R$ is solution to \eqref{eqequwithkc} where $\la_{\mu, j}$ denotes the Lagrange multiplier associated to $u_{\mu, j}$. In addition 
\begin{equation}
E(u_{\mu, j},\G) = c_{k_j-1}\in \Big[\inf_{u  \in\s_{k_j-1}^\perp \cap \B^{M_2}} E(u, \G), \sup_{u \in \s_{k_j}}E(u, \G)\Big].
\end{equation}

Reasoning as in the proof of Lemma \ref{lemvaluesepe} we get 
\begin{align} \label{ajout10}
		\inf_{u\in\s_{k_j -1}^\perp \cap \B^{M_2}}E(u,\G)
		& \geq g(\la_{k_j}(\G)\mu)- \ell^{\frac{2-p}2}\mu^\frac p2. \nonumber \\
		& =  \frac{1}{2}\la_{k_j} \mu - K \mu^{\frac{p+2}{4}}(\la_{k_j} \mu)^{\frac{p-2}{4}} - \ell^{\frac{2-p}2}\mu^\frac p2.
	\end{align}
	Clearly also,
	$$\sup_{u \in {\s_{k_j}}}E(u, \G) \leq \frac{1}{2}\la_{k_j}\mu.$$
	Thus,
	\begin{equation*}
E(u_{\mu, j},\G) \in \Big[\frac{1}{2}\la_{k_j} \mu - K \mu^{\frac{p+2}{4}}(\la_{k_j} \mu)^{\frac{p-2}{4}} - \ell^{\frac{2-p}2}\mu^\frac p2, \frac{1}{2}\la_{k_j}\mu \Big].
\end{equation*}
We then have
\begin{align}\label{ajout11}
\frac{E(u_{\mu, j}, \G)}{\mu} \to \frac{1}{2}\la_{k_j}, \quad \mbox{as } \mu \to 0^+.
\end{align}
Now, from Lemma \ref{lemgn} and since $\int_{\G}|u'_{\mu,j}|^2 dx \leq \rho^*$,  we get
\begin{align}\label{ajout13}
\frac{\int_{\G}|u_{\mu,j}|^p dx}{\mu} \to 0, \quad \mbox{as } \mu \to 0^+.
\end{align}
Combining \eqref{ajout11} and \eqref{ajout13}, it follows 
\begin{align*}
\frac{\int_{\G}|u_{\mu,j}'|^2dx}{\mu} \to \frac{1}{2}\la_{k_j}, \quad \mbox{as } \mu \to 0^+ .
\end{align*}
Finally, from \eqref{eqequwithkc}, we deduce that
\begin{align*}
-\la_{\mu, j} = \frac{\int_{\G}|u_{\mu,j}'|^2dx}{\mu} - \frac{\int_{\G}|u_{\mu,j}|^p dx}{\mu} \to \la_{k_j}, \quad \mbox{as } \mu \to 0^+.
\end{align*}
Since $\int_{\G}|u_{\mu,j}'|^2 dx \to 0$ as $\mu \to 0^+$, this shows that $\la_{k_j}$ is a bifurcation point and ends the proof.
\end{proof}

\end{document}